\documentclass[11pt]{amsart}
\usepackage{hyperref}
\usepackage{color}
\usepackage{graphicx}
\usepackage{microtype}

\newtheorem{thm}{Theorem}[section]
\newtheorem{lem}[thm]{Lemma}
\newtheorem{prop}[thm]{Proposition}
\newtheorem{cor}[thm]{Corollary}

\newtheorem{quest}[thm]{Question}

\newcommand{\newtheorembox}[2]{\newtheorem{x#1}[thm]{#2}
    \newenvironment{#1}{\begin{x#1}}{\qed\end{x#1}}}

\theoremstyle{definition}
\newtheorem{defn}[thm]{Definition}
\newtheorembox{ex}{Example}
\newtheorembox{rmk}{Remark}

\newcommand{\arxiv}[1]{\href{http://arxiv.org/abs/#1}{arXiv:#1}}
\newcommand{\bff}{\mathbf f}
\newcommand{\cA}{\mathcal A}

\newcommand{\cD}{\mathcal D}
\newcommand{\cP}{\mathcal P}
\newcommand{\defi}[1]{\emph{#1}}
\newcommand{\degree}{\operatorname{deg}}

\newcommand{\isom}{\cong}
\newcommand{\link}{\operatorname{link}}

\newcommand{\NS}{\operatorname{NS}}
\newcommand{\Pic}{\operatorname{Pic}}

\newcommand{\QQ}{\mathbb Q}
\newcommand{\ridge}{\operatorname{ridge}}
\newcommand{\RR}{\mathbb R}

\newcommand{\td}{\operatorname{td}}
\newcommand{\tors}{\mathrm{tors}}

\newcommand{\ZZ}{\mathbb Z}
\renewcommand{\div}{\operatorname{div}}

\title{Combinatorial tropical surfaces}
\author{Dustin Cartwright}
\address{Department of Mathematics \\ University of Tennessee \\ 227 Ayres Hall
         Knoxville, TN 37996}
\email{cartwright@utk.edu}

\begin{document}

\begin{abstract}
We study the combinatorial properties of $2$-dimensional tropical complexes. In
particular, we prove tropical analogues of the Hodge index theorem
and Noether's formula. In addition, we introduce algebraic
equivalence for divisors on tropical complexes of arbitrary dimension.
\end{abstract}

\maketitle

\section{Introduction}

In recent years, a number of results from algebraic curves have inspired
tropical analogues for finite graphs and tropical curves, such as the tropical
Riemann-Roch, Abel-Jacobi, and Torelli
theorems~\cite{baker-norine,mikhalkin-zharkov,brannetti-melo-viviani}. In this
paper, we look at combinatorial analogues of basic results on algebraic
surfaces. Our combinatorial setting is that of tropical complexes, as introduced
in~\cite{cartwright-complexes}. A tropical complex consists of both a
combinatorial topological space, specifically a $\Delta$-complex, together with
some additional integers, which give a theory of divisors generalizing that on
graphs to higher dimensions. For most of this paper we will work with
2-dimensional tropical complexes, which we will refer to as tropical surfaces.

Our first result is an analogue of the Hodge index theorem for the intersection
product on a tropical surface.
In addition to the axioms of a tropical complex, we need the
topological assumption that the underlying $\Delta$-complex is locally connected through
codimension~$1$, meaning that the link of each vertex is connected.

\begin{thm}\label{thm:hodge}
Let $\Delta$ be a tropical surface which is
locally connected through codimension~$1$. Then the intersection pairing on
$\NS(\Delta) \otimes_\ZZ \QQ$ is a non-degenerate bilinear form
whose matrix has at most one positive eigenvalue.
\end{thm}

The N\'eron-Severi group $\NS(\Delta)$ from Theorem~\ref{thm:hodge} refers to
the group of Cartier divisors on~$\Delta$ up to algebraic equivalence. Algebraic
equivalence on tropical complexes is introduced in this paper and
serves as a coarser relationship than linear equivalence of
divisors~\cite[Sec.~4]{cartwright-complexes}. Algebraic equivalence arises
naturally from a long exact sequence analogous to the exponential sequence for
complex manifolds. Note that algebraic equivalence and the tropical exponential
sequence are defined for tropical complexes of arbitrary dimension, unlike the
majority of the results in this paper that only apply to surfaces.

While smooth proper algebraic surfaces are always projective and thus have a
divisor with positive self-intersection, the same is not true for tropical
surfaces. In fact, even for tropical surfaces which are locally
connected through codimension~$1$, N\'eron-Severi group can be trivial (see
Example~\ref{ex:hopf} for details). Thus, we suggest that having a divisor with positive
self-intersection is an analogue for tropical surfaces of a projectivity or
K\"ahler hypothesis on a complex surface. For example, the following gives us a
tropical analogue of the Jacobian of a surface under such a hypothesis.
\begin{thm}\label{thm:jacobian}
Let $\Delta$ be a tropical surface which is locally connected through
dimension~$1$ and has a divisor~$D$ such that $\deg D^2 > 0$. Then, the group of
algebraically trivial divisors on~$\Delta$ modulo linear equivalence
of~$\Delta$ has the structure of a real torus $(\RR/\ZZ)^b$ where $b =
\dim_{\RR} H^1(\Delta, \RR)$.
\end{thm}

Our next result is a definition of the second Todd class and a tropical analogue
of Noether's formula. Recall that the second Todd class on a projective
algebraic surface~$S$
can be defined as $\frac{1}{12}(K_S^2 + c_2(S))$, where $K_S$ is the canonical
divisor and $c_2(S)$ is the second Chern class of~$S$. Then, Noether's formula
asserts that the degree of the second Todd class equals the holomorphic Euler
characteristic of~$S$. We construct a second Todd class not just on tropical
surfaces, but on weak tropical surfaces, for which one of the axioms of a
tropical surface is dropped. The second Todd class of a weak tropical surface is
a formal sum of the vertices, with coefficients in $\frac{1}{12}\ZZ$.
\begin{prop}\label{prop:noether}
If $\Delta$ is a weak tropical surface, then the degree of its second
Todd class $\td_2(\Delta)$ equals the topological Euler characteristic
$\chi(\Delta)$.
\end{prop}

Our definition of the second Todd class is compatible with the local invariant
from the affine linear Gauss-Bonnet theorem of Kontsevich and
Soibelman~\cite[Sec.~6.5]{kontsevich-soibelman} and with
Shaw's definition of the second Chern class on tropical
manifolds~\cite[Sec.~3.2]{shaw}. Moreover, the second Todd class has an
interpretation as a local invariant at any point of a multiplicity-free tropical
variety.

The rest of the paper is organized as follows. In Section~\ref{sec:local}, we
look at the local structure of tropical surfaces and prove an analogue of the
maximum modulus principle. In Section~\ref{sec:alg}, we define the tropical
exponential sequence and algebraic equivalence using a cohomological
interpretation of the Picard group of a tropical complex.
In Section~\ref{sec:pairing}, we study the intersection
pairing on a tropical surface, proving Theorems~\ref{thm:hodge}
and~\ref{thm:jacobian}. Section~\ref{sec:noether} defines the second
Todd class and proves Noether's formula.

\subsection*{Acknowledgments}

Throughout this project, I've benefited from my conversations with Matt Baker,
Spencer Backman, Alex Fink, Christian Haase, Paul Hacking, June Huh, Eric Katz,
Madhusudan Manjunath, Kristin Shaw, Farbod Shokrieh, Bernd Sturmfels, Yu-jong
Tzeng, and Josephine Yu. I'd especially like to thank Sam Payne for his many
insightful suggestions and thoughtful comments on an earlier draft of this
paper. I was supported by the National Science Foundation award number
DMS-1103856.

\section{Local theory}\label{sec:local}

In this section, we will study the local structure of tropical surfaces. We
first recall the definition of a tropical complex~\cite[Sec.~2]{cartwright-complexes}, in the $2$-dimensional case of interest in
this paper. A tropical complex is built from a $\Delta$-complex is a
combinatorial class of topological spaces as in \cite[Sec.~2.1]{hatcher} or
\cite[Def.~2.44]{kozlov}. In the 2-dimensional case, we will refer to the cells
of dimensions $0$, $1$, and~$2$ as the \defi{vertices},
\defi{edges}, and \defi{facets}, respectively.

Since $\Delta$-complexes allow faces of
a single simplex to be identified with each other, an edge may contain only $1$
vertex, but we will write \defi{endpoint} for the endpoints of an edge before
such an identification. The link, denoted $\link_\Delta(v)$ or
$\link_\Delta(e)$, encodes the local structure of~$\Delta$ around a vertex~$v$
or an edge~$e$ of~$\Delta$~\cite[p. 31]{kozlov}. The link of an edge is a finite set and the link of
a vertex~$v$ is a graph whose vertices correspond to edges~$e$ together with an
identification of~$v$ as an endpoint of~$e$.

\begin{defn}
A \defi{$2$-dimensional weak tropical complex} or \defi{weak tropical
surface}~$\Delta$
is a finite, connected $\Delta$-complex, consisting of cells of dimension at
most~$2$, together with integers $\alpha(v, e)$
for every endpoint~$v$ of an edge~$e$ such that,
for each edge~$e$ with endpoints $v$ and~$w$, we have an equality:
\begin{equation}\label{eq:constraint}
\alpha(v, e) + \alpha(w, e) = \deg(e),
\end{equation}
where $\deg(e)$ is the cardinality of $\link_\Delta(e)$.

At each vertex~$v$ of a weak tropical surface~$\Delta$, we construct a
\defi{local intersection matrix} $M_v$ whose rows and columns are indexed by the
vertices of $\link_\Delta(v)$ and with the entry corresponding to $t, u \in
\link_\Delta(v)_0$ equal to:
\begin{equation*}
(M_{v})_{t,u} = \begin{cases}
\#\{\mbox{edges between $t$ and~$u$ in $\link_\Delta(e)$}\} &
\mbox{if } t \neq u \\
-\alpha(w, e) + 2 \cdot \#\{\mbox{loops at $t$ in $\link_\Delta(e)$}\} &
\mbox{if } t = u,
\end{cases}
\end{equation*}
where, in the second case, $e$ is the edge corresponding to $t=u$ and $w$ is
the endpoint of~$e$ not identified with~$v$.
A weak tropical surface~$\Delta$ is called a \defi{2-dimensional tropical
complex} or \defi{tropical surface} if $M_v$
has exactly one positive eigenvalue for every vertex~$v$ of~$\Delta$.
\end{defn}

A \defi{PL function} on a weak tropical surface is continuous function~$\phi$
such that $\phi$ is piecewise linear with integral slopes on each facet, if we
identify the facet with a unimodular simplex in $\RR^2$. Section~4
of~\cite{cartwright-complexes} introduces a general framework by which a linear
combination of line segments is associated to a PL function on a weak tropical
surface. The simplest case is a piecewise linear function~$\phi$ which is linear
in the ordinary sense on all the simplices of~$\Delta$. Then, the divisor
of~$\phi$ is a linear combination of the edges of~$\Delta$ where the
coefficient of an edge~$e$ with endpoints $v$ and~$v'$ is:
\begin{equation}\label{eq:linear-on-simplices}
-\alpha(v, e) \phi(v) - \alpha(v', e) \phi(v') + \sum_{(f, w) \in
\link_\Delta(e)} \phi(w), 
\end{equation}
where the summation is over facets~$f$ with vertices~$w$ such that the edge
of~$f$ opposite~$w$ is identified with~$e$. See Definition~2.5 and
Proposition~4.8 in~\cite{cartwright-complexes} for the general version of this
theorem. A \defi{linear function} on a weak tropical surface is a PL function
whose divisor is trivial, and we will use the phrase \defi{linear in the
ordinary sense} to distinguish functions which are linear according to the
identification of each $k$-dimensional simplex with a unimodular simplex in
$\RR^k$.

We'll now look at a local version of~(\ref{eq:linear-on-simplices}) for
functions~$\phi$ which are linear on simplices in a neighborhood of a vertex~$v$
of~$\Delta$. Since $\phi$ isn't defined globally, we can't use its values at
vertices like in~(\ref{eq:linear-on-simplices}). Moreover, adding a
constant to~$\phi$ doesn't affect the divisor, so our formula will be in terms of
the slopes of~$\phi$. In particular, we suppose there are $r$ vertices in
$\link_\Delta(v)$, which we order.
We record the slopes of $\phi$ in a vector $\bff \in \ZZ^r$ whose $i$th
entry~$\bff_i$ is the slope along the $i$th incidence of an edge~$e$ of~$\Delta$
to~$v$, where $e$ is taken to have length~$1$. If $\phi$ were a globally defined function, then $\bff_i$ would be
$\phi(w) - \phi(v)$
where $w$ is the other endpoint other than~$v$ of the edge corresponding to the
$i$th vertex of $\link_\Delta(v)$. We can also record the multiplicity of
$\div(\phi)$ along each edge incident to~$v$ in a vector of the same size
as~$\bff$. Then,
\begin{lem}\label{lem:local-intersection-matrix}
Let $\phi$ be a function on a neighborhood of~$v \in \Delta$, which is linear in
the ordinary sense on each simplex, and let $\bff$ encode the slopes of~$\phi$
as above. Then the coefficients of $\div(\phi)$ in a
neighborhood of~$v$ are given by $M_v \bff$, where $M_v$ is the local
intersection matrix at~$v$.
\end{lem}

\begin{proof}
This follows from Proposition~4.6 in \cite{cartwright-complexes}.
\end{proof}

We now wish to generalize the matrix~$M_v$ and
Lemma~\ref{lem:local-intersection-matrix} to functions which are only piecewise
linear on the simplices of~$\Delta$ and to points other than vertices. To do so,
we need a finer decomposition of the neighborhood of a point than the one given
by the simplices of~$\Delta$, for which we have the following definition.

\begin{defn}\label{def:local-cone-complex}
Let $p$ be a point of weak tropical complex~$\Delta$ and then a
\defi{local cone complex} of~$\Delta$ at~$p$ is a subdivision of a neighborhood
of~$p$ into unimodular cones, in the following sense. For any facet~~$f$
containing~$p$, we identify a neighborhood of~$p$ in~$f$ with a neighborhood of
the origin in the cones $\RR_{\geq 0}^2$, $\RR \times \RR_{\geq 0}$, or $\RR^2$
if $p$ is a vertex, contained in the interior of an edge, or contained in the
interior of a facet, respectively. Moreover, we assume that the identification
preserves the integral structure, meaning that the differences between the
vertices of~$f$ generate the lattice of integral vectors in $\RR^2$. We then
subdivide the cone $\RR_{\geq 0}^2$, $\RR \times \RR_{\geq 0}$, or $\RR^2$ into
any fan with pointed unimodular cones. The local cone complex is formed by
gluing the cones for each facet containing $p$ along their common edges, so that
a neighborhood of~$p$ in~$\Delta$ is identified with a neighborhood of the
origin in the cone complex.
\end{defn}

\begin{figure}
\includegraphics{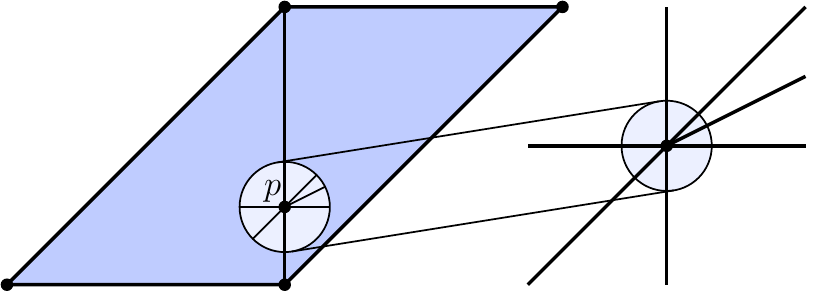}
\caption{An example of a local cone complex for a point~$p$ on the boundary
between two facets. The local cone complex consists of the union of the seven
$2$-dimensional cones at right which subdivide the two facets of the simplicial
complex.}\label{fig:local-cones}
\end{figure}

\begin{ex}\label{ex:local-cones}
Figure~\ref{fig:local-cones} shows an example of the local cone complex on a
tropical complex consisting of two facets. The two facets containing~$p$ are
divided into three and four $2$-dimensional cones, respectively, and this
subdivision passes to a neighborhood of the point in the tropical complex.
\end{ex}

\begin{lem}\label{lem:local-cone-complex}
If $D$ is any set of edges with rational slopes on a weak tropical
surface~$\Delta$, then at any point $p \in \Delta$, there exists a local cone
complex~$\Sigma$ such that in a neighborhood of~$p$, $D$ is supported on the
rays of~$\Sigma$.
\end{lem}

\begin{proof}
In a neighborhood of~$p$, on each $2$-simplex, the edges of~$D$ define a
subdivision of a $2$-dimensional cone, either $\RR^2_{\geq 0}$, $\RR_{\geq 0} \times
\RR$, or $\RR^2$, into
rational cones. Thus, all that's left to do is to refine this subdivision into
one that is unimodular, which is combinatorially equivalent to toric resolution
of singularities, and toric resolutions of singularities always
exist~\cite[Sec.~2.6]{fulton-toric}.
\end{proof}

An important case for Lemma~\ref{lem:local-cone-complex} is when $D$ is the
divisor associated to a PL function~$\phi$. Then, $\phi$ will be linear in a
neighborhood of the origin on each cone of the local cone complex~$\Sigma$, 
and so we can extend~$\phi$ by linearity to a function on the whole cone
complex, which we also denote by $\phi$. Now let $e_1, \ldots, e_r$ denote the
rays of $\Sigma$. We can record the slopes of~$\phi$ in a
vector~$\bff \in \ZZ^r$, where the $i$th entry, corresponding to a ray~$e_i$, is
$\phi(v_i) - \phi(0)$, where $v_i$ is the minimal integral vector along $e_i$ in
the local cone complex. The divisor of~$\phi$ is supported on the rays of the
cone complex, so we can also record its coefficients as a vector in $\ZZ^r$
whose entries correspond to the multiplicity along a given ray.

If we identify with $\ZZ^r$ with the group of all PL functions which are linear
on the cones of~$\Sigma$, then we have a linear map $\ZZ^r
\rightarrow \ZZ^r$ recording the encoding of divisors of the PL function. We
will refer to the matrix representation of of this map as $M_{p,\Sigma}$, which
will serve as a generalization of the local intersection matrix~$M_v$.

\begin{lem}\label{lem:link}
The off-diagonal entries of~$M_{p,\Sigma}$ coincide with the adjacency matrix of
$\Sigma$, meaning that if $i \neq j$, the $(i,j)$-entry of $M_{p,\Sigma}$
records the number of cones containing both rays $e_i$ and~$e_j$.
\end{lem}

\begin{proof}
We let $e_i$ and $e_j$ be two rays of $\Sigma$ and let $\sigma$ be a cone
containing both of them. Since $\sigma$ is unimodular, we can choose coordinates
on~$\sigma$ such that $e_i = (1,0)$ and $e_j = (0,1)$ and because the integral
structure on~$\sigma$ agrees with that on the facets of~$\Delta$, we can use the
same coordinates the corresponding region of~$\Delta$. Then, the PL
function~$\phi_i$ which is~$1$ on $e_i$ and $0$ on the other rays can be defined
by the $x$ coordinate on~$\sigma$. Thus, along~$e_j$, $\phi_i$ is defined by
$\max\{x, 0\}$, so by the construction of the divisor of a PL function
\cite[Prop.~4.5(iv)]{cartwright-complexes}, the contribution of~$\sigma$ to the
multiplicity of~$\phi_i$ along $e_j$ is~$1$. By linearity of
multiplicities~\cite[Prop.~4.5(i)]{cartwright-complexes}, the total multiplicity
of~$\phi_i$ along~$e_j$ is the number of such cones~$\sigma$, which is what we
wanted to show.
\end{proof}

\begin{lem}\label{lem:local-matrix}
If $\Delta$ is a tropical surface, then $M_{p,\Sigma}$ has exactly one positive
eigenvalue.
\end{lem}

\begin{proof}
We first show that for any point~$p$, there exists some local cone
complex~$\Sigma$ for which the lemma is true and then we'll show that the
statement is independent of the choice of~$\Sigma$. If $p$ is a vertex
of~$\Delta$, then we can choose $\Sigma$ induced by the facets containing~$p$,
in which case $M_{p,\Sigma} = M_{p}$ by
Lemma~\ref{lem:local-intersection-matrix} and $M_p$ has exactly one positive
eigenvalue by the definition of a tropical complex.

If $p$ is in the interior of an edge, then for each facet containing $p$, we
choose to subdivide $\RR_{\geq 0} \times \RR$ by adding a single ray, say $e_i =
\RR_{\geq 0} \times \{0\}$ in some coordinates. The first coordinate function on
$\RR_{\geq 0} \times \RR$ is
$1$ on the integral point $(1,0)$ of the ray~$e_i$, but vanishes on the other
rays. This function is linear on $\RR_{\geq 0} \times \RR$, so its multiplicity
along $e_i$ is $0$, so the $i$th diagonal entry of $M_{p,\Sigma}$ is $0$.
Moreover, Lemma~\ref{lem:link} gives us the off-diagonal entries of
$M_{p,\Sigma}$. Therefore, if we assume that the rays of~$\Sigma$ are ordered
with the subdividing rays first, we have:
\begin{equation}\label{eq:edge-local-matrix}
M_{p,\Sigma} = \begin{pmatrix}
0 & \cdots & 0 & 1 & 1 \\
\vdots & & \vdots & \vdots & \vdots \\
0 & \cdots & \mathbf{0} & \mathbf{1} & 1 \\
1 & \cdots & \mathbf{1} & \mathbf{a} & 0 \\
1 & \cdots & 1 & 0 & b
\end{pmatrix},
\end{equation}
for some values of $a$ and $b$. By the construction of linear
functions~\cite[Const.~4.2]{cartwright-complexes}, linear functions in a
neighborhood of~$p$ form a free group of rank $d$, which denotes the cardinality
of $\link_\Delta(e)$, so
$M_{p,\Sigma}$ has kernel of dimension~$d$ and so only $2$ non-zero eigenvalues.
On the other hand, the bold $2 \times 2$ submatrix
of~(\ref{eq:edge-local-matrix}) is not semidefinite for any value of~$a$, so the
$2$ non-zero eigenvalues must have opposite sign. Therefore, $M_{p,\Sigma}$ has
exactly one positive eigenvalue.

If $p$ is in the interior of a facet, we choose
$\Sigma$ by dividing $\RR^2$ along the rays generated by $(1,0)$, $(0,1)$,
and~$(-1,-1)$. Then one can compute
\begin{equation*}
M_{p,\Sigma} =
\begin{pmatrix} 1 & 1 & 1 \\ 1 & 1 & 1 \\ 1 & 1 & 1 \end{pmatrix},
\end{equation*}
which has exactly one positive eigenvalue.

Now, it only remains to show that the number of positive eigenvalues of
$M_{p,\Sigma}$ is independent of the local cone complex. Any two cone
complexes~$\Sigma$ and~$\Sigma'$ have a common refinement~$\Sigma''$, which
simultaneously gives a unimodular subdivision of each cone of~$\Sigma$ and of
each cone of~$\Sigma'$. By the theory of minimal models for toric varieties,
$\Sigma''$ can be formed from $\Sigma$ by a series of subdivisions, meaning
replacing a cone spanned by minimal integral rays $v$ and~$w$ with two cones,
one spanned by $v$ and $v+w$ and the other spanned by $v+w$ and~$w$. Thus, it
suffices to prove that if $\Sigma'$ is a subdivision of~$\Sigma$, then
$M_{p,\Sigma}$ and $M_{p,\Sigma'}$ have the same number of positive eigenvalues.

We can order the rays of $\Sigma$ such that $v$ and~$w$ are last, and then we
assign the following variables to the lower right corner of~$M_{p,\Sigma}$:
\begin{equation*}
M_{p,\Sigma} = 
\begin{pmatrix}
* & * & * \\
* & a_v & b \\
* & b & a_w
\end{pmatrix}.
\end{equation*}
For $\Sigma'$, we put the new ray $v + w$ last and then we claim that:
\begin{equation*}
M_{p,\Sigma'} = 
\begin{pmatrix}
* & * & * & 0 \\
* & a_v - 1 & b - 1 & 1 \\
* & b-1 & a_w - 1 & 1 \\
0 & 1 & 1 & -1
\end{pmatrix}.
\end{equation*}
The off-diagonal entries of $M_{p,\Sigma'}$ are justified by
Lemma~\ref{lem:link}. The diagonal entries can be justified by
considering the PL functions which are linear on $\Sigma$ corresponding to the
vectors $(0, \ldots, 1, 0)$ and $(0, \ldots, 0, 1)$. When we look at either of
these functions on~$\Sigma'$, the their value at~$v+w$ is~$1$, so the encodings
of these functions on~$\Sigma'$ are $(0, \ldots, 1, 0, 1)$ and $(0, \ldots, 0,
1, 1)$. Thus, 
\begin{align*}
M_{p,\Sigma'} (0, \ldots, 1, 0, 1)^T &= (0, \ldots, a_v, b, 0) \\
M_{p,\Sigma'} (0, \ldots, 0, 1, 1)^T &= (0, \ldots, b, a_w, 0),
\end{align*}
which determines the diagonal entries.

Now, we apply the following change of variables:
\begin{equation*}
\begin{pmatrix}
I & 0 & 0 & 0 \\
0 & 1 & 0 & 1 \\
0 & 0 & 1 & 1 \\
0 & 0 & 0 & 1 \\
\end{pmatrix} M_p' \begin{pmatrix}
I & 0 & 0 & 0 \\
0 & 1 & 0 & 0 \\
0 & 0 & 1 & 0 \\
0 & 1 & 1 & 1
\end{pmatrix} = \begin{pmatrix}
* & * & * & 0 \\
* & a_w & b & 0 \\
* & b & a_u & 0 \\
0 & 0 & 0 & -1
\end{pmatrix},
\end{equation*}
to get a block diagonal matrix where the upper left block is~$M_{p,\Sigma}$
and whose lower right block is a negative entry.
Thus, $M_{p,\Sigma}$ and $M_{p,\Sigma'}$ have the same number positive
eigenvalues, which completes the proof of the lemma.
\end{proof}

We now turn to applications of the matrix $M_{p,\Sigma}$.
First, we recall that Section~4 of~\cite{cartwright-complexes} defined a hierarchy of different types
of divisors, all of which are integral sums of line segments on a tropical
surface. Cartier divisors are locally defined by a PL function, $\QQ$-Cartier
divisors are those such that some multiple is Cartier, and Weil divisors are
$\QQ$-Cartier except for a set of codimension at least~$3$. Thus, on a weak
tropical surface, Weil divisors coincide with $\QQ$-Cartier divisors and so we
will call them just \defi{divisors}. These divisors can be defined by
\defi{rational PL functions} which are functions such that some positive
multiple is a PL function. In addition, tropical complexes have
curves, which are formal sums of line segments such that the restriction of an
affine linear function is affine linear~\cite[Def.~5.2]{cartwright-complexes}.
However, on surfaces, curves and divisors have the same dimension and the
$\QQ$-Cartier condition on divisors coincides with the balancing condition on
curves.

\begin{prop}\label{prop:curve-divisor}
On a weak tropical surface, a formal sum of segments is a curve if and only if
it is a divisor.
\end{prop}

\begin{proof}
Let $C$ be a formal sum of segments in~$\Delta$. Since the conditions for being
a curve and a divisor are both local, it suffices to check that they are
equivalent in a neighborhood of a point~$p$. By
Lemma~\ref{lem:local-cone-complex}, we can find a local cone complex~$\Sigma$
at~$p$ such that $C$ is supported on the rays of the subdivision. We represent
the coefficients of~$C$ by a vector~$\mathbf c$ whose entries are indexed by
these rays. Supposing that there exists a rational PL function~$\phi$ defining
$C$, then $\phi$ will be linear on each cone of~$\Sigma$. Let $\bff$ be the rational vector
recording the slope of~$\phi$ along each of the rays so that $\mathbf c =
M_{p,\Sigma} \bff$. Thus, $C$ is a divisor in a neighborhood of~$p$ if and only
if the coefficient vector~$\mathbf c$ is in the image of the matrix $M_p$.

Likewise, a PL function~$\phi$ is linear in a neighborhood of~$p$ if and only if
the corresponding vector~$\bff$ is in the kernel of~$M_{p,\Sigma}$. If so, then
the degree of $\phi$ restricted to~$C$ at~$p$ is given by the slopes of~$\phi$
along the edges of~$C$, which is computed by the dot product $\bff \cdot \mathbf
c$. Thus, $C$ is a curve if and only if for every vector~$\bff$ in the kernel
of~$M_{p,\Sigma}$, the vector~$\mathbf c$ is orthogonal to~$\bff$. The kernel of
a matrix is the orthogonal complement of its rows, so the orthogonal complement
of the kernel of~$M_{p,\Sigma}$ is the image of the transpose $M_{p,\Sigma}^T$.
However, $M_{p,\Sigma}$ is symmetric, so $C$ is a curve if and only if $\mathbf
c$ is in the image of~$M_{p,\Sigma}$, which we've already shown to be equivalent
to $C$ being a divisor.
\end{proof}

Since every divisor is a curve and vice versa, the intersection of two divisors
can be computed two different ways, depending on which is considered as a
divisor and which as a curve. However, these two methods produce the same
result.

\begin{prop}\label{prop:intersection-sym}
The intersection product on a weak tropical surface is symmetric.
\end{prop}

\begin{proof}
Let $C$ and $D$ be two divisors and fix a point $p$ in their intersection. Let
$\phi$ and~$\psi$ be the functions defining $C$ and~$D$ respectively in a
neighborhood of~$p$. We can find a local cone complex~$\Sigma$ at~$p$ for $C
\cup D$ and then both $C$ and~$D$ will be supported on the rays of~$\Sigma$. We
let $\mathbf c$ and~$\mathbf d$ be the vectors representing the coefficients of
$C$ and $D$ in a neighborhood of~$p$ and let $\bff$ and $\mathbf g$ be the
vectors of the slopes of $\phi$ and~$\psi$
respectively. Considering $C$ as a curve and
$D$ as a divisor,
their local intersection product is the dot product $\mathbf c \cdot \mathbf g$ and we have the
equalities:
\begin{equation*}
\mathbf c \cdot \mathbf g = 
\mathbf c^T \mathbf g = \mathbf f^T M_{p,\Sigma}^T \mathbf g =
\mathbf f^T M_{p,\Sigma} \mathbf g = \mathbf f^T \mathbf d =
\mathbf f \cdot \mathbf d,
\end{equation*}
which is the computation of the intersection product if we reverse the roles of
$C$ and~$D$.
\end{proof}

A proper algebraic variety has no non-constant regular functions. We will now
prove a combinatorial analogue for tropical surfaces, and we note that it will
not hold for weak tropical surfaces. Moreover, we also need the following
combinatorial condition on the underlying topology of~$\Delta$.

\begin{defn}
A tropical surface~$\Delta$ is \defi{locally connected through codimension~$1$}
if $\link_\Delta(v)$ is connected for each vertex~$v$.
\end{defn}

\noindent Connectivity through codimension~1 is a well-known concept in tropical
geometry because it is a property of the tropicalization of any irreducible
variety over an algebraically closed field~\cite{bjsst,cartwright-payne}.

In the following proposition, we say that a divisor is \defi{effective} if all
of its coefficients are non-negative. If we regard PL functions whose divisor is
effective as analogous to regular functions in algebraic geometry or holomorphic
functions of a complex variable, then the following is an analogue of the
maximum modulus principle in complex analysis.

\begin{prop}\label{prop:no-regular-local}
Let $\Delta$ be a tropical surface which is locally connected
through codimension~$1$ and let $\phi$ be a PL function on a connected open set
$U \subset \Delta$. If $\phi$ achieves a maximum on~$U$ and the divisor
of~$\phi$ is effective, then $\phi$ is constant.
\end{prop}

\begin{proof}
Let $p$ be a point at which $\phi$ achieves its maximum. By
Lemma~\ref{lem:local-cone-complex}, let $\Sigma$ be a local cone complex
of~$\Delta$ at~$p$ such that $\phi$ is linear on its cones. We choose a
sufficiently large number~$c$, such that all entries of $M_{p,\Sigma} + cI$ are
non-negative, where $I$ is the identity matrix. In the sense of the
Perron-Frobenius theorem, a square matrix~$M$ is irreducible if there does not
exist a non-empty, proper subset~$J$ of the indices such that the entries
$M_{i,j}$ are zero whenever $i \not\in J$ and $j \in
J$~\cite[Def.~2.1.2]{berman-plemmons}. By Lemma~\ref{lem:link},
such a set~$J$ for $M_{p,\Sigma}$ would correspond
to a non-trivial connected component in $\link_\Sigma(p)$, which would
contradict the local connectivity hypothesis. Therefore, by the Perron-Frobenius theorem,
$M_{p, \Sigma} + cI$ has a unique eigenvector~$\mathbf w$, with positive
entries, whose eigenvalue~$\lambda$ has maximal norm among all eigenvalues of
$M_{p,\Sigma} + cI$ \cite[Thm.~2.1.3b, 2.1.1b]{berman-plemmons}. Thus, $\mathbf
w$ is an eigenvector of $M_{p,\Sigma}$ with eigenvalue $\lambda - c$, which is
greater than all other eigenvalues of~$M_{p,\Sigma}$. Since $\Delta$ is a
tropical surface, Lemma~\ref{lem:local-matrix} shows that $M_{p, \Sigma}$
has a unique positive eigenvalue, and so $\lambda - c$ must be that positive
eigenvalue.

Now let $\bff$ be the vector containing the outgoing slopes of~$\phi$ at $p$.
Then, $M_{p,\Sigma} \bff$ contains the coefficients of the divisor of~$\phi$ in
a neighborhood of~$p$, and we've assumed these coefficients to be non-negative.
Every entry of~$\mathbf w$ is positive, and so $\mathbf w^T M_{p,\Sigma} \bff =
(\lambda - c) \mathbf w^T \bff$ is non-negative, and since
$\lambda - c$ is positive, the entries of $\mathbf w^T \bff$ are also non-negative. On
the other hand, $\phi$ is maximal at~$p$, so the entries of~$\bff$ are
non-positive and the only way for $\mathbf w^T \bff$ to be non-negative is if
$\bff$ is identically zero. Thus, $\phi$ is constant in a neighborhood of~$p$.

Therefore, the subset of~$U$ where $\phi$ achieves its maximum is open, and
since it is closed by continuity of~$\phi$ and non-empty by assumption, $\phi$
must be constant on~$U$.
\end{proof}

\begin{cor}\label{cor:no-regular}
Let $\Delta$ be a tropical surface which is locally connected through
codimension~$1$. If $\phi$ is a PL function on~$\Delta$ whose associated divisor
is effective, then $\phi$ is constant.
\end{cor}

\begin{proof}
By Proposition~\ref{prop:no-regular-local}, it suffices to show that $\phi$
achieves its maximum. However, a finite $\Delta$-complex is compact, so $\phi$
always achieves its maximum on~$\Delta$.
\end{proof}

Both Proposition~\ref{prop:no-regular-local} or
Corollary~\ref{cor:no-regular} are false for weak tropical surfaces by the
Example~\ref{ex:no-hodge} below. Since Corollary~\ref{cor:no-regular} is a
natural analogue of a result from algebraic geometry, and it doesn't reference
intersections or eigenvalues, the fact that it fails for weak tropical
surfaces shows the importance of our condition on local intersection matrices
in our definitions of tropical surfaces and tropical complexes.

\begin{figure}
\begin{centering}\includegraphics{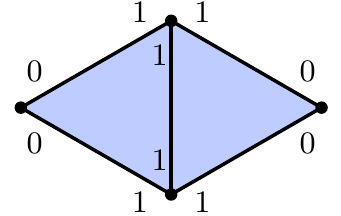}\par\end{centering}
\caption{This complex is a weak tropical surface which is not a tropical surface
because the local intersection matrices at the leftmost and rightmost vertices
are both negative semidefinite. The complex also
violates the conclusion of Corollary~\ref{cor:no-regular} in that the horizontal
coordinate function is a non-constant PL function with trivial divisor.}
\label{fig:no-hodge}
\end{figure}

\begin{ex}\label{ex:no-hodge}
Let $\Delta$ be the weak tropical surface illustrated in
Figure~\ref{fig:no-hodge}. This complex violates is not a tropical complex
because of the vertices on the left and right. The
PL function which is $0$ on the left vertex, $1$ on the middle vertices, and $2$
on the rightmost vertex, and linear on every simplex is obviously non-constant,
but its divisor is trivial, and thus effective.
\end{ex}

\section{Algebraic equivalence of divisors}\label{sec:alg}

Two divisors on a tropical surface are defined to be linearly equivalent if
their difference is the divisor of a PL function. We now define a coarsening of
this equivalence relation, which is algebraic equivalence of divisors. In this
paper, our primary application for algebraic equivalence is to reduce questions
about intersection theory of tropical surfaces in Section~\ref{sec:pairing} to divisors
supported on the edges. However, since the proofs generalize easily, we work
with weak tropical complexes of arbitrary dimension throughout this section.

Similar to a weak tropical surface, a weak tropical complex of arbitrary
dimension~$n$ consists of both a finite, connected $\Delta$-complex whose
simplices have dimension at most~$n$ and some integers $\alpha(v,r)$, for which
we refer to~\cite[Def.~2.1]{cartwright-complexes} for
details. Similar to weak tropical surfaces, a weak tropical complex has local
intersection matrices associated to each $(n-2)$-dimensional simplex, which each
have exactly one positive eigenvalue for a tropical complex. However, we will
not need the local intersection matrix in this section and will only work with
weak tropical complexes.

A PL function on a weak tropical complex is defined as being piecewise linear
with integral slopes on each simplex, analogously to the surface
case. We will
denote the sheaf of PL functions by~$\cP$. Any PL function~$\phi$ defines a
divisor $\div(\phi)$ a formal sum of polyhedral subsets of~$\Delta$ by
\cite[Prop.~4.5]{cartwright-complexes}. The linear functions on open subsets of
$\Delta$ are the PL functions~$\phi$ such that $\div(\phi)$ is trivial, and we
denote the sheaf of linear functions by $\cA$.

A Cartier divisor on~$\Delta$ is a formal sum of polyhedra which is locally the
divisor of a PL function. Thus, a Cartier divisor can be given by an open cover
$\{U_i\}$ together with a PL function $\phi_i$ on each $U_i$ such that for each
pair of indices $i$ and $j$, $\div(\phi_i)$ and $\div(\phi_j)$ agree on $U_i
\cap U_j$. This condition is equivalent to requiring $\phi_i\vert_{U_i \cap U_j}
- \phi_j \vert_{U_i \cap U_j}$ to always be a linear function, so a Cartier
divisor is equivalent to a global section of the quotient sheaf $\cP / \cA$.
Cartier divisors which are linearly equivalent to zero are those defined by a
global section of $\cP$. From the long-exact sequence in cohomology, we have the
following sheaf-theoretic description of the group of Cartier divisors modulo
linear equivalence, which we call the Picard group~$\Pic(\Delta)$.

\begin{prop}\label{prop:h-1}
The Picard group $\Pic(\Delta)$ is isomorphic to
$H^1(\Delta, \cA)$.
\end{prop}

\begin{proof}
The quotient sheaf $\cP / \cA$ gives us the following long exact sequence in
sheaf cohomology:
\begin{equation*}
0 \rightarrow H^0(\Delta, \cA) \rightarrow
H^0(\Delta, \cP) \rightarrow H^0(\Delta, \cP / \cA) \rightarrow
H^1(\Delta, \cA) \rightarrow H^1(\Delta, \cP) \rightarrow
\end{equation*}
Thus, by the discussion above it will be sufficient to show that $H^1(\Delta, \cP)$
is trivial. We will show that all higher sheaf cohomology of~$\cP$ is trivial by
showing that it is a soft sheaf~\cite[Thm.~II.4.4.3]{godement}. Recall that
$\cP$ is called a soft sheaf if for any closed set $Z \subset \Delta$, the map
$H^0(\Delta, \cP) \rightarrow H^0(Z, \cP)$ is surjective, where $H^0(Z, \cP)$ is
the direct limit of $H^0(U, \cP)$ as $U$ ranges over all open sets~$U$
containing~$Z$.

We let $\phi$ be a function in $H^0(Z, \cP)$, which can be represented by a PL
function on some open set $U \supset Z$, which we also denote~$\phi$. We choose
an open set $V \supset Z$ with polyhedral boundary and whose closure is
contained in~$U$. Since $\phi$ is piecewise linear, it is bounded on~$U$ and we
let $C$ be a constant less than the minimum of~$\phi$. For any integer~$N$, we define the
function~$\phi_N$ on~$U$ by
\begin{equation*}
\phi_N(x) = \max\{C, \phi(x) - N d_1(\overline V, x)\},
\end{equation*}
where $d_1(\overline V,x)$ denotes the minimum distance, in the $L^1$~metric,
between $x$ and the closure of~$V$. It is clear that $\phi_N$ agrees with~$\phi$
on~$V$, and therefore, they have the same image in $H^0(\Delta, Z)$. Moreover,
$\phi_N$ is a PL function. For sufficiently large~$N$, we can
ensure that $\phi_N$ takes the value~$C$ at every point in a neighborhood of
$\Delta \setminus U$. Thus, we can extend $\phi_N$ by~$C$ to a function on all
of~$\Delta$, which completes the proof that $\cP$ is soft.
\end{proof}

Following~\cite[Sec.~5.1]{mikhalkin-zharkov}, we can use sheaf cohomology to
define an analogue of the exponential sequence. The sheaf of locally constant
real-valued functions, which we denote~$\RR$, is a subsheaf of $\cA$, and we
will denote the quotient sheaf $\cA / \RR$ as $\cD$.
In~\cite{mikhalkin-zharkov}, $\cD$ is called the cotangent sheaf.
We are interested in the following section of the long exact sequence in
cohomology:
\begin{equation}\label{eq:exp}
\rightarrow
H^0(\Delta, \cD) \rightarrow H^1(\Delta, \RR) \rightarrow
H^1(\Delta, \cA) \rightarrow H^1(\Delta, \cD) \rightarrow
H^2(\Delta, \RR) \rightarrow
\end{equation}
Note that the cohomology of the sheaf~$\RR$ agrees with simplicial cohomology
on~$\Delta$, which justifies the notation.
Also, if $\Delta$ is a tropical surface, locally connected through
codimension~$1$, then $H^0(\Delta, \cA)$ vanishes by
Corollary~\ref{cor:no-regular}, and thus first map of~(\ref{eq:exp}) is
injective.

The exact sequence~(\ref{eq:exp}) has a
striking analogy to the following piece of the long exact sequence coming from
the exponential sequence on a complex variety~$X$:
\begin{equation*}
\rightarrow
H^1(X, \ZZ) \rightarrow H^1(X, \mathcal O_X) \rightarrow
H^1(X, \mathcal O_X^*) \rightarrow H^2(X, \ZZ) \rightarrow
H^2(X, \mathcal O_X) \rightarrow
\end{equation*}
In both exact sequences, the middle term is isomorphic to the
Picard group of the weak tropical complex or variety.

Following this analogy, we define the \defi{Chern class} of an element of the
Picard group $H^1(\Delta, \cA)$ to be its image in $H^1(\Delta, \cD)$. For
tropical curves, $H^1(\Delta, \cD)$ is isomorphic $\ZZ$ and the Chern class of a
divisor records its degree. For surfaces and higher-dimensional tropical
complexes, $H^1(\Delta, \cD)$ can be higher rank, and there can also be an
obstruction in $H^2(\Delta, \RR)$, so that not every element of $H^1(\Delta,
\cD)$ is the Chern class of a divisor. We will refer to the image of the Picard
group $H^1(\Delta, \cA)$ in $H^1(\Delta, \cD)$ as the \defi{N\'eron-Severi
group} $\NS(\Delta)$. Two divisors are \defi{algebraically equivalent} if their
difference is a Cartier divisor with trivial Chern class.

\begin{figure}
\includegraphics{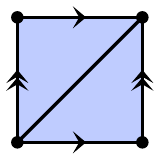}
\caption{The triangulation of the $2$-dimensional torus used in
Example~\ref{ex:torus}. A torus is formed by identifying the horizontal edges
with each other and the vertical edges with each other as indicated
by the arrow labeling.}\label{fig:torus}
\end{figure}

\begin{ex}\label{ex:torus}
This example is the $2$-dimensional case of the theory of tropical Abelian
varieties discussed in~\cite[Section~5.1]{mikhalkin-zharkov}.
Let $\Delta$ be the triangulation of a $2$-dimensional torus depicted in
Figure~\ref{fig:torus} with $\alpha(v,e) = 1$ for all endpoints~$v$ of all
edges~$e$. The sheaf $\cD$ is isomorphic to the sheaf of locally constant
functions valued in~$\ZZ^2$, by taking a linear function to its
derivatives in the $x$ and~$y$ directions, and therefore $H^1(\Delta, \cD) \isom
\ZZ^4$.

For $\Delta$, the long exact sequence (\ref{eq:exp}) is:
\begin{equation*}
0 \rightarrow
\ZZ^2 \rightarrow \RR^2 \rightarrow (\RR/\ZZ)^2 \oplus \ZZ^3 \rightarrow
\ZZ^4 \rightarrow \RR \rightarrow
\end{equation*}
The summand~$(\RR/\ZZ)^2$ of the Picard
group corresponds to algebraically trivial divisors, each of which is linearly
equivalent to a unique divisor of the form
$[\pi_1^{-1}(s)] - [\pi_1^{-1}(0)] + [\pi_2^{-1}(t)] - [\pi_2^{-1}(0)]$,
where $\pi_1$ and~$\pi_2$ are the two coordinate projections
from~$\Delta$ to the cycle of length~$1$ and
$s$ and~$t$ are arbitrary points on the $1$-cycle and
$0$ is its vertex. The N\'eron-Severi group is~$\ZZ^3$, whose
generators can be taken to be the three edges in Figure~\ref{fig:torus}.
Thus, $\NS(\Delta)$ is a proper
subgroup of $H^1(\Delta, \cD)$, and the map to $H^2(\Delta, \RR)
\isom \RR$ is non-trivial.
\end{ex}

We will now show that algebraic equivalence preserves intersection numbers,
which is an analogue of the classical fact that algebraic equivalence implies
numerical equivalence~\cite[p.~374]{fulton-intersection}. The following
generalizes Proposition~5.6 from~\cite{cartwright-complexes}, which dealt with
linear equivalence.

\begin{prop}\label{prop:intersect-alg}
If $D$ and $D'$ are algebraically equivalent and $C$ is a curve, then the
degrees of $D \cdot C$ and $D' \cdot C$ are equal.
\end{prop}

\begin{proof}
Since the intersection number is linear in the divisor, it is sufficient to show
that if $D$ is algebraically trivial, then $D \cdot C$ has degree~$0$.
By definition, if $D$ is algebraically trivial then it is in the image of
$H^1(\Delta, \RR)$ in the exponential sequence~(\ref{eq:exp}). This means that
the differences between the local equations for~$D$ from one chart to the next
are locally constant functions. However, the intersection with $C$ depends only
on the slope, so for any segment of $C$ for which the defining equations are
linear, the contribution will be opposite at the two ends. Thus, the total
degree will be $0$, as we wanted to show.
\end{proof}

Divisors on weak tropical complexes are defined to be generated by arbitrary polyhedral subsets, but it's
computationally most convenient to work with divisors supported on the ridges
of~$\Delta$, which we will call \defi{ridge divisors}. We define
$\Pic_{\ridge}(\Delta)$ to be the group of Cartier ridge divisors modulo linear
equivalence. We also define $\cA_\ZZ$ to be the subsheaf of~$\cA$ consisting of
linear functions such that, on each simplex, the linear extension of the
function takes on integral values at the vertices.

\begin{prop}\label{prop:ridge-divisors}
The group $\Pic_{\ridge}(\Delta)$ is
isomorphic to $H^1(\Delta, \cA_\ZZ)$.
\end{prop}

\begin{proof}
To prove this, we repeat the construction from before
Proposition~\ref{prop:h-1}, but for ridge divisors. Thus, we let $\cP_\ZZ$ be
the sheaf of piecewise linear
functions which are linear on each simplex and whose extensions take an integral
value at each vertex. Then, the global sections of $\cP_\ZZ / \cA_\ZZ$
correspond to Cartier ridge divisors and those which are linearly trivial are
defined by global sections of~$\cP_\ZZ$. We have a long exact sequence in
cohomology:
\begin{equation*}
\rightarrow
H^0(\Delta, \cP_\ZZ) \rightarrow
H^0(\Delta, \cP_\ZZ / \cA_\ZZ) \rightarrow
H^1(\Delta, \cA_\ZZ) \rightarrow
H^1(\Delta, \cP_\ZZ) \rightarrow,
\end{equation*}
from which
the proposition will follow if we can prove that $H^1(\Delta, \cP_\ZZ)$
vanishes.

In fact, we will show that $H^i(\Delta, \cP_\ZZ) = 0$ for all $i > 0$ via a
computation of \v Cech cohomology~\cite[Thm.~II.5.10.1]{godement}. Thus, we have
an open cover $\{U_i\}$ which we assume to be sufficiently refined such that
each $U_i$ and only intersects a single simplex~$s_i$ together with
simplices which have $s_i$ as a face and each such intersection is convex.
Moreover, if $U_i$ are small enough, then $U_i \cap U_j$ is non-empty only if
$s_i$ is a face of~$s_j$ or vice versa. As usual, we denote the intersection of
open sets $U_{i_1} \cap \cdots \cap U_{i_k}$ by $U_{i_1, \ldots, i_k}$. Each
non-empty $U_{i_1, \ldots, i_k}$ only intersects a single simplex $s_{i_1,
\ldots, i_k}$ and the simplices containing $s_{i_1, \ldots, i_k}$, where
$s_{i_1, \ldots, i_k}$ is the largest of the simplices $s_{i_1}, \ldots,
s_{i_k}$. A section of $\cP_\ZZ$
on $U_{i_1, \ldots, i_k}$ is equivalent to the values of the linear extension at
each vertex. Thus, the group of \v Cech $k$-cocycles of $\cP_\ZZ$ is the direct
product:
\begin{equation}\label{eq:cech-cocycles-1}
C^k\big(\{U_i\}; \cP_\ZZ\big) = 
\prod_{i_1 < \ldots < i_k} \prod_{v} \ZZ,
\end{equation}
where $v$ ranges over the vertices of the parametrizing simplex of $s_{i_1,
\ldots, i_k}$ and of the simplices containing it, up to identifications which
also contain $s_{i_1, \ldots, i_k}$.

We now reinterpret the factors of the \v Cech $k$-cocycles
in~(\ref{eq:cech-cocycles-1}) coming from a fixed
vertex~$v$. Let $K_v$ be the cone over $\link_\Delta(v)$. Then there exists a
natural map $\pi_v \colon K_v \rightarrow \Delta$ sending the cone point to~$v$,
and where the preimage of a point in the
interior of a simplex~$s$ consists distinct points
corresponding to the vertices of the parametrizing simplex of~$s$ identified
with~$v$. Thus, the number of times~$v$ shows up in the \v Cech $k$-cocycles is
equal to the number of connected components of $\pi_v^{-1}(U_{i_1, \ldots,
i_k})$. We write $\ZZ$ for the sheaf of locally constant integer-valued
functions on~$K_v$ and then we can rewrite the group of \v Cech $k$-cocycles
from~(\ref{eq:cech-cocycles-1}) as:
\begin{equation}\label{eq:cech-cocycles-2}
\prod_{v \in \Delta_0}  \prod_{i_1 < \cdots < i_k}
\ZZ\big(\pi_v^{-1}( U_{i_1, \ldots, i_k})\big) = \prod_{v \in \Delta_0}
C^k\big(\{\pi^{-1}(U_i)\}; \ZZ\big).
\end{equation}
These equalities are compatible with the boundary maps and refinements, so
if we take the limit over all refinements of $\{U_i\}$, then the limit of the
cohomology of the right hand side of~(\ref{eq:cech-cocycles-2}) computes the
$\ZZ$-cohomology of $K_v$. However, the integer cohomology of $K_v$ vanishes
for $k > 0$ because $K_v$ is contractible. Therefore, the cohomology
$H^i(\Delta, \cP_{\ZZ})$ vanishes for $i > 0$, which completes the proof.
\end{proof}

Having the cohomological interpretation for $\Pic_{\ridge}(\Delta)$ allows us
to have an analogue for ridge divisors of the exponential
sequence~(\ref{eq:exp}):
\begin{equation}\label{eq:exp-ridge}
\rightarrow
H^0(\Delta, \cD) \rightarrow H^1(\Delta, \ZZ) \rightarrow
\Pic_{\ridge}(\Delta) \rightarrow
H^1(\Delta, \cD) \rightarrow H^2(\Delta, \ZZ) \rightarrow
\end{equation}
The quotient $\cA_\ZZ / \ZZ$ is isomorphic to the
sheaf~$\cD$ from before because $\cA$ and $\cA_{\ZZ}$ only differ by the
allowable constant terms of the linear functions.
Composing the map
$\Pic(\Delta) \rightarrow H^1(\Delta, \cD)$ from (\ref{eq:exp}) with the last
map $H^1(\Delta, \cD) \rightarrow H^2(\Delta, \ZZ)$ of (\ref{eq:exp-ridge}), we
get a map $\Pic(\Delta) \rightarrow H^2(\Delta, \ZZ)$.

\begin{prop}\label{prop:ridge-obstruction}
Let $D$ be an element of the Picard group of~$\Delta$. Then, the
image of~$D$ under the above map lies in the torsion subgroup $H^2(\Delta,
\ZZ)_{\tors}$ and is trivial if and only if $D$ is algebraically equivalent to a
ridge divisor.
\end{prop}

\begin{proof}
From the inclusions $\ZZ \rightarrow \RR$ and $\cA_{\ZZ} \rightarrow \cA$, we
have a map to the exponential sequence~(\ref{eq:exp}) from that for ridge
divisors~(\ref{eq:exp-ridge}). By exactness of the former, the image of~$D$ in
$H^2(\Delta, \RR)$ is trivial, so it lies in the kernel of $H^2(\Delta, \ZZ)
\rightarrow H^2(\Delta, \RR)$, which is the torsion subgroup, therefore proving
the first claim. Then, the second part of the claim then follows from the
exactness of~(\ref{eq:exp-ridge}).
\end{proof}

\begin{cor}\label{cor:multiple-ridge}
If $\Delta$ is a weak tropical complex, then there exists a positive integer $m$
such that for any Cartier divisor~$D$ on~$\Delta$, the multiple $mD$ is
algebraically equivalent to a ridge divisor.
\end{cor}

\begin{cor}\label{cor:finitely-generated}
If $\Delta$ is a weak tropical complex, then $NS(\Delta)$ is finitely generated.
\end{cor}

\begin{proof}
By Proposition~\ref{prop:ridge-obstruction}, we have an exact sequence
\begin{equation*}
\Pic_{\ridge}(\Delta) \rightarrow \NS(\Delta) \rightarrow
H^2(\Delta, \ZZ)_{\tors}
\end{equation*}
Since $\Delta$ is a finite complex, both $\Pic_{\ridge}(\Delta)$ and
$H^2(\Delta, \ZZ)_{\tors}$ are finitely generated, so $\NS(\Delta)$ is also
finitely generated.
\end{proof}

\begin{ex}\label{ex:octahedron-quotient}
\begin{figure}
\includegraphics{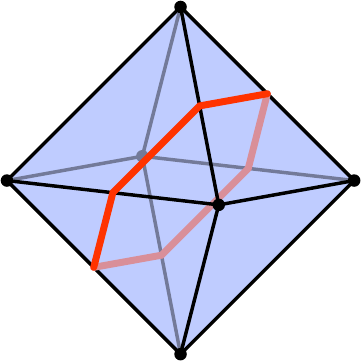}
\caption{If we quotient by octahedron by the involution formed by negating the
coordinates, then the image of the pictured red divisor is not algebraically
equivalent to any linear combination of ridge divisors.}
\label{fig:octahedron-quotient}
\end{figure}
Let $\widetilde\Delta$ be the boundary of the octahedron with all structure
constants
$\alpha(v, e)$ set to~$1$. We take $\Delta$ to be the quotient of
$\widetilde\Delta$ by the involution of the octahedron which takes each point to
its opposite. Let $D$ be the divisor on~$\Delta$ which is the quotient of the
cycle shown in red in Figure~\ref{fig:octahedron-quotient}. We claim that $D$ is
not
algebraically equivalent to any linear combination of ridge divisors. Let $e$ be
the sum of the $3$~edges of~$\Delta$ which intersect~$D$. Each
point of intersection has multiplicity~$1$ and so the intersection number of~$D$
with~$e$ is~$3$. On the other hand, the intersection number of any ridge with
$e$ is either $2$ or~$0$, so no integral combination of ridges can be
have intersection number~$3$ with~$e$.
However, since $H^2(\Delta, \ZZ) = \ZZ/2$, we know that twice $D$ must be
algebraically equivalent to a ridge divisor. Explicitly, $2D$ is linearly
equivalent to the sum of the $3$ edges which it doesn't intersect.
\end{ex}

In this paper, ridge divisors will be useful primarily for computational
purposes. However, on a graph, the group of linear equivalence classes of
ridge divisors which are algebraically trivial has the same cardinality as the
set of spanning trees by Kirchhoff's matrix tree theorem. It is natural to
wonder whether there is a generalization of the matrix tree theorem to
higher-dimensional weak tropical complexes.
\begin{quest}
Is there a combinatorial interpretation for the order of the group of
algebraically trivial ridge divisors modulo linear equivalence?
\end{quest}
\noindent In Example~\ref{ex:hopf}, we'll see that this group may be infinite.

\section{Intersection pairing}\label{sec:pairing}

We now return to tropical surfaces where we study the intersection pairing on
divisors, which is the bilinear map which takes a pair of divisors to the degree
of their intersection product.
The essential result from the previous
section is Proposition~\ref{prop:ridge-obstruction} which allows us to work with
ridge divisors.

We start by constructing a matrix which constitutes a global version of the
matrices $M_{p, \Sigma}$ from Section~\ref{sec:local}, except that by working
only with ridge divisors, there is no need for the choice of a local cone
complex. We let $M_\Delta$ denote the block diagonal matrix formed by taking all
the matrices $M_v$ as $v$ ranges over the vertices of~$\Delta$. Thus, $M_\Delta$
is a symmetric matrix whose rows and columns are indexed by a vertex~$v$
of~$\Delta$ and a vertex of the link at~$v$, or, equivalently, an edge~$e$
of~$\Delta$ together with an endpoint of~$e$. Throughout this section, we
write $\QQ^N$ for space of vectors whose entries are indexed by the pair of a
vertex~$v$ and an edge~$e$ with $v$ as one endpoint.

In the same way that, in Section~\ref{sec:local}, $M_{p,\Sigma}$ computed the
divisor of a PL function which is linear on~$\Sigma$, $M_\Delta$ can compute the
divisor of a global PL function~$\phi$, which is linear in the ordinary sense on
each simplex. We encode $\phi$ in a vector $\bff_\phi \in \QQ^N$ where
$(\bff_\phi)_{v, e} = \phi(w) - \phi(v)$, where $w$ is the endpoint of~$e$ other
than~$v$. Then, by Lemma~\ref{lem:local-intersection-matrix}, $\mathbf d = M_\Delta \bff_{\phi}$ encodes the divisor of~$\phi$,
but redundantly because the coefficient of $e$ is given as $\mathbf d_{v,e} =
\mathbf d_{w,e}$.

Moreover, given any Cartier divisor~$D$ on~$\Delta$, there is, by definition, a
local defining equation~$\phi_v$ for~$D$ around every vertex~$v$. If we set
$\bff_{v, e}$ equal to the slope of~$\phi$ moving away from $v$ along~$e$,
similar to the construction in Section~\ref{sec:local}, then the coefficients
of~$D$ are encoded in $\mathbf d = M_\Delta \bff$, where the coefficient of an
edge~$e$ is given by $\mathbf d_{v,e} = \mathbf d_{w,e}$. Moreover, this
equality $\mathbf d_{v,e} = \mathbf d_{w,e}$ encodes the compatibility
conditions for a divisor so that if $\bff$ is any integral or rational vector
such that $M_\Delta \bff$ satisfies these equalities, then $\bff$ encodes the
local defining equations of a Cartier or $\QQ$-Cartier divisor, respectively.

\begin{prop}\label{prop:matrix-intersections}
Let $M_\Delta$ be the matrix described above. If $\bff$ and $\bff'$ are vectors
encoding the local defining equations of Cartier divisors $D$ and $D'$
respectively, then $\bff^T M_\Delta \bff'$ equals the total degree of the
product $D \cdot D'$.
\end{prop}

\begin{proof}
Because of the block diagonal structure of~$M_\Delta$, the product $\bff^T
M_\Delta \bff'$ computes the sum of local contributions for each vertex~$v$.
Those contributions are the multiplicity at~$v$ of the intersection product~$D
\cdot D'$ as in the proof of Proposition~\ref{prop:intersection-sym}.
\end{proof}

Section~\ref{sec:alg} presented algebraic equivalence in a sheaf-theoretic way,
but we now give an explicit description of the map $H^1(\Delta, \ZZ)
\rightarrow \Pic(\Delta)$ from the exponential sequence~(\ref{eq:exp-ridge}) in
terms of simplicial cohomology.
Let $\gamma$ be a simplicial 1-cochain on~$\Delta$, so a function from the
oriented edges of~$\Delta$ to~$\ZZ$, such that reversing orientation on an edge
negates the value of~$\gamma$. We will construct a Cartier divisor from~$\gamma$
by giving the local defining equations as a vector~$\bff_\gamma \in \ZZ^N$. For
any vertex $v$ and an incidence of an edge~$e$ to~$v$, we set
$(\bff_\gamma)_{v,e} = \gamma(e)$, where we consider $e$ to be oriented away
from~$v$.

\begin{lem}\label{lem:cocycle}
If $\gamma$ is a 1-cocycle, then $\bff_\gamma$ as above defines a Cartier
divisor. Moreover, the class of this Cartier divisor agrees with the image
of the cohomology class of~$\gamma$ under the map $H^1(\Delta, \ZZ) \rightarrow
\Pic_{\ridge}(\Delta)$ from~(\ref{eq:exp-ridge}).
\end{lem}

\begin{proof}
The vector~$\bff$ defines a system of local equations~$\phi_v$ in a
neighborhood~$U_v$ of each vertex~$v$ of~$\Delta$. To show that these define a
Cartier divisor, we need to show that the difference of functions is linear and
to show it is in the image of $H^1(\Delta, \ZZ)$, we need to show that it is
locally constant. Let $v$ and~$w$ be two vertices with an edge~$e_{vw}$ between
them, oriented from~$v$ to~$w$. Let $f$ be a facet containing~$e_{vw}$ with $u$
its other vertex and $e_{vu}$ and $e_{wu}$ its edges, both oriented toward $u$.
The local equations on~$f$ take the values:
\begin{align*}
\phi_v(v) &= 0 & \phi_v(w) &= \gamma(e_{vw}) & \phi_v(u) &= \gamma(e_{vu}) \\
\phi_w(v) &= -\gamma(e_{vw}) & \phi_w(w) &= 0 & \phi_w(u) &= \gamma(e_{wu})
\end{align*}
Thus, $\phi_v - \phi_w$ will be constant on the interior of~$f$ and equal to
$\gamma(e_{vw})$ by the cocycle condition $\gamma(e_{vw}) - \gamma(e_{vu}) +
\gamma(e_{wu}) = 0$. Moreover, the \v Cech cocycle which has value
$\gamma(e_{vw})$ on the component of $U_v \cap U_w$ containing~$e$ defines the
same cohomology class as $\gamma$, which finishes the proof.
\end{proof}

Lemma~\ref{lem:cocycle} shows that the cocycle condition on~$\gamma$ is
sufficient to define a divisor. We also have the following converse to that
implication:

\begin{lem}\label{lem:criterion-alg-trival}
If $\gamma$ is a 1-cochain on a weak tropical complex such that $\bff_\gamma$
defines a Cartier divisor, then $\gamma$ is a cocycle.
\end{lem}

\begin{proof}
We let $\gamma$ be a 1-cochain on~$\Delta$ such that the vector $\bff_\gamma$
defines a Cartier divisor. From $\bff_\gamma$, we have a local function~$\phi_v$
on an open neighborhood $U_v$ of each vertex~$v$. For $\bff_\gamma$ to define a
Cartier divisor means that for each edge~$e$, the multiplicity of $\phi_v$ along
$e$ agrees with the multiplicity of $\phi_w$, where $v$ and $w$ are the
endpoints of~$e$. Before computing these multiplicities, we set up some
notation. If
$t$ is a vertex in $\link_\Delta(e)$, then $t$ corresponds to an identification
of $e$ with one of the edges of a facet~$f$, and we denote the other edges
of~$f$ as $e_{vt}$ and $e_{wt}$, containing $v$ and~$w$, respectively.
Then, using Lemma~\ref{lem:local-intersection-matrix}, the multiplicity of~$\phi_v$ along~$e$
is:
\begin{equation}\label{eq:mult-v}
-\alpha(w, e) \gamma(e) + \sum_{t \in \link_\Delta(e)} \gamma(e_{vt}),
\end{equation}
where $e$ and $e_{vt}$ are both oriented away from~$v$. Similarly, with the same
orientation on~$e$, and with $e_{wt}$ oriented away from~$w$, the multiplicity
of~$\phi_w$ will be:
\begin{equation}\label{eq:mult-w}
\alpha(v, e) \gamma(e) + \sum_{t \in \link_\Delta(e)} \gamma(e_{wt}).
\end{equation}
Using the identity~(\ref{eq:constraint}) from the definition of a weak tropical
complex, the difference of (\ref{eq:mult-w}) and~(\ref{eq:mult-v}) is:
\begin{equation}\label{eq:diff-mults}
(\deg e) \gamma(e) + \sum_{t \in \link_\Delta(e)} \gamma(e_{wt}) -
\gamma(e_{vt})
 = \sum_{t \in \link_\Delta(e)} \gamma(e) + \gamma(e_{wt}) - \gamma(e_{vt}),
\end{equation}
and this sum will be zero since $\gamma$ defines a Cartier divisor.

Now, fix an orientation on every edge and facet so that we can represent the
1-cochains and 2-cochains of~$\Delta$ as vectors in $\ZZ^E$ and $\ZZ^F$, where
$E$ and $F$ are the numbers edges and facets of~$\Delta$, respectively. Then,
the simplicial coboundary map is defined by an $F \times E$ matrix~$B$
and (\ref{eq:diff-mults}) is either the entry corresponding to~$e$ of
$B^T B \gamma$
or its negative, depending on the orientation we fixed for~$e$ agrees with
the orientation used in the previous paragraph. Thus, $B^TB \gamma = 0$, but since $\gamma$ and $B$ have
real entries, $\gamma^T B^T B \gamma = 0$ implies $B \gamma = 0$. Therefore,
$\gamma$ is in the kernel of the coboundary map, so $\gamma$ is a cocycle by
definition, which is what we wanted to show.
\end{proof}

Up to torsion, the algebraically trivial divisors on a weak tropical surface
from the previous section can also be characterized by their intersection
numbers~(cf.\ \cite[19.3.1(ii)]{fulton-intersection}). In other words, the
intersection pairing on $\NS(\Delta) \otimes_{\ZZ} \QQ$ is non-degenerate, which
is one part of Theorem~\ref{thm:hodge}.

\begin{prop}\label{prop:numerical-algebraic}
Let $D$ be a divisor on a weak tropical surface~$\Delta$. Then $\deg D \cdot D'
= 0$ for all divisors $D'$ if and only if $mD$ is algebraically trivial for
some $m \geq 0$.
\end{prop}

\begin{proof}
By Proposition~\ref{prop:intersect-alg}, if $mD$ is algebraically trivial, then
$\deg mD \cdot D'  = 0$, and since the intersection product is linear, $D \cdot
D'$ is also degree $0$ for all divisors~$D'$.

On the other hand, suppose that $\deg D \cdot D' = 0$ for all divisors $D'$. By
replacing $D$ and $D'$ with multiples, we can assume that they are Cartier
divisors. By scaling further, we can assume that they are ridge divisors by
Corollary~\ref{cor:multiple-ridge}.
As discussed above, for $D'$ Cartier, the local defining equations of~$D'$ are
encoded in a vector~$\bff$ where $M_\Delta \bff$ satisfies certain equalities.
If $v$ and~$w$ are the endpoints of an edge~$e$, we
let $\mathbf c_e$ be the vector defined by $(\mathbf c_e)_{v,e} = 1$,
$(\mathbf c_e)_{w,e} = -1$, and zeros elsewhere. Then, the compatibility
condition for $\bff$ defining the same coefficient on~$e$ at either endpoint is
characterized by $\mathbf c_e^T M_\Delta \bff = 0$.

Let $C$ be the vector subspace of~$\QQ^N$ generated by the $\mathbf c_e$. Thus,
if we consider $M_{\Delta}$ as defining a bilinear form on $\QQ^N$, then the
vectors $\bff$ defining the local equations of $\QQ$-Cartier divisors are the
orthogonal complement $C^{\perp}$, taken with respect to $M_{\Delta}$.
Therefore, the numerically trivial divisors are those in $(C^\perp)^{\perp}$,
which is equal to $C + \ker M_{\Delta}$ by linear algebra. Since the divisors
associated to a function in $\ker M_{\Delta}$ is zero, we can assume that $D$ is
defined by a vector $\bff \in C$.

Let $m$ be a positive integer such that the entries of $m\bff$ are integers. If
set $\gamma(e) = m\bff_{v,e}$, where $v$ is the initial endpoint of an oriented
edge~$e$, then we get a simplicial 1-cochain by the definition of~$C$. Then, $m
\bff$ agrees with $\bff_{\gamma}$, and Lemma~\ref{lem:criterion-alg-trival}
shows that $\gamma$ is a 1-cocycle, so $mD$ is algebraically trivial, which is
what we wanted to show.
\end{proof}

We now turn to the proof of the tropical Hodge index theorem. The crux is the
following lemma.

\begin{lem}\label{lem:master}
Let $\Delta$ be a tropical surface which is locally connected through
codimension~$1$. If the intersection pairing on $\Pic_{\ridge}(\Delta)
\otimes_{\ZZ} \QQ$ has kernel of dimension~$k$ and has $m$ positive eigenvalues,
then $m + k \leq 1$.
\end{lem}

\begin{proof}
Let $D_1, \ldots, D_k$ be Cartier ridge divisors whose classes span the kernel
of the intersection pairing on $\Pic_{\ridge}(\Delta) \otimes_{\ZZ} \QQ$. Let
$\mathbf g_i \in \QQ^N$ be the vector recording the local defining equations
of~$D_i$. We
number the vertices $v_1, \ldots, v_n$, and let $\phi_1, \ldots, \phi_n$ be the
PL functions which are linear in the ordinary sense on each facet and such
that $\phi_{i}(v_j) = \delta_{ij}$, where $\delta_{ij}$ is the Kronecker delta
function. Then, let
$\bff_i \in \QQ^N$ be the vector which encodes $\phi_i$. Since our encoding
records the slopes of the PL function and $\sum_{i=1}^n \phi_i$ is identically
$1$, we have a relation $\sum_{i = 1}^n \bff_i
= 0$ among our vectors.
Now let $H_1, \ldots, H_m$ be Cartier ridge divisors whose intersection form is
positive definite, with $\mathbf h_i \in \QQ^N$ the vector encoding the defining
equations of~$H_i$.

Let $L$ be the vector subspace of $\QQ^N$ spanned by the $\bff_i$ and $\mathbf
g_j$. We claim that $M_\Delta L \subset \QQ^N$ has dimension $n + k - 1$.
Suppose we have a relation among the generators of this vector space, say:
\begin{equation*} M_\Delta (c_1 \bff_1 + \ldots + c_n \bff_n + d_1 \mathbf g_1 +
\cdot + d_k \mathbf g_k) = 0 \end{equation*} Since multiplying by~$M_\Delta$
computes the divisor associated to the local functions, this means that $d_1 D_1
+ \cdots + d_k D_k = \div(\phi)$ where $\phi = -c_1 \phi_1 - \cdots - c_n
\phi_n$. However, we've assumed that the $D_j$ are linearly independent in
$\Pic_{\ridge}(\Delta)$, so that means that $d_1 = \cdots = d_k = 0$. Thus,
$\phi$ is a linear function, so
by Corollary~\ref{cor:no-regular} and 
since $\Delta$ is locally connected through codimension~$1$,
$\phi$ is constant, meaning that $c_1 = \cdots = c_n$.
Therefore, the only relation among the $M \bff_i$ and $M \mathbf g_j$ comes from
the relation $\sum_{i=1}^n \bff_i = 0$ already noted, and so their span has
dimension $n + k - 1$ as desired.

We can therefore find dual vectors $\bff_{1}^*, \ldots \bff_{n-1}^*$ and
$\mathbf g_1^*, \ldots, \mathbf g_{k}^*$ such that 
\begin{equation*}
(\mathbf f_i^*)^T M_\Delta \mathbf g_j = 0 \qquad
(\mathbf f_i^*)^T M_\Delta \bff_j = \delta_{ij} \qquad
(\mathbf g_i^*)^T M_\Delta \mathbf g_j = \delta_{ij},
\end{equation*}
where $\delta_{ij}$ is again the Kronecker delta function. If we restrict the
bilinear form defined by $M_\Delta$ to the vectors we've defined in the
following order:
\begin{equation*}
\bff_{1}, \ldots, \bff_{n-1}, \mathbf g_1, \ldots, \mathbf g_k,
\bff_1^*, \ldots, \bff_{n-1}^*, \mathbf g_1^*, \ldots, \mathbf g_k^*,
\mathbf h_1, \ldots, \mathbf h_m,
\end{equation*}
then the resulting pairing has the matrix form:
\begin{equation}\label{eq:restricted-form}
\begin{pmatrix}
0 & I & 0 \\
I & 0 & 0 \\
0 & 0 & M_H
\end{pmatrix},
\end{equation}
where the identity blocks have size $n + k -1$ and $M_H$ is the positive
definite intersection matrix for the $H_i$. Thus, the
matrix~(\ref{eq:restricted-form}) has $n + k
- 1 + m$ positive eigenvalues.

On the other hand, $M_\Delta$ is a block diagonal matrix whose blocks $M_v$ each
have
exactly $1$ positive eigenvalue, so $M_\Delta$ has exactly $n$ positive
eigenvalues. Since restricting a bilinear form to a subspace can only decrease
the number of positive eigenvalues, we get $k + m -1 \leq 0$, which is what we
wanted to show.
\end{proof}

\begin{proof}[Proof of Theorem~\ref{thm:hodge}]
We've already shown that the intersection pairing on $\NS(\Delta) \otimes_\ZZ
\QQ$ is non-degenerate in Proposition~\ref{prop:intersect-alg}. By
Corollary~\ref{cor:multiple-ridge}, the
map from $\Pic_{\ridge(\Delta)} \otimes_{\ZZ} \QQ$ to $\NS(\Delta) \otimes_{\ZZ}
\QQ$ is surjective. The intersection pairing on the former has at most one
positive eigenvalue and so the same is true for the intersection pairing on
$\NS(\Delta) \otimes_{\ZZ} \QQ$, which completes the proof.
\end{proof}

Unlike the case of smooth proper algebraic surfaces, but like complex analytic
surfaces, tropical surfaces do not always have a divisor with positive
self-intersection, so the intersection pairing can be negative definite. The
tropical complex in the following example has no algebraically non-trivial
divisors, so any intersection product of divisors is zero by
Proposition~\ref{prop:intersect-alg}.

\begin{figure}
\includegraphics{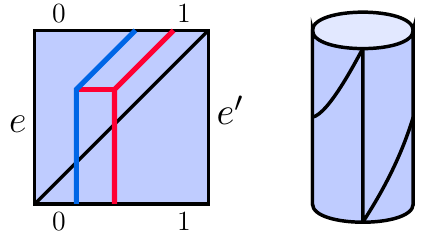}
\caption{Our tropical complex will be formed by gluing the edges $e$ and $e'$ of
the square on the left to form the cylinder pictured on the right. The structure
constants on $e = e'$ and the diagonal edge are all~$1$ and those of the top and
bottom edges are as indicated by the numbers at either end of those edges. One
can check that these define a tropical
complex. In Example~\ref{ex:hopf}, we show how, for the resulting tropical
complex~$\Delta$, the inclusion of $H^0(\Delta, \cD)$ into $H^1(\Delta, \RR)$ is
not a lattice and the N\'eron-Severi group is~$\RR$. The divisor corresponding
to a real number~$\lambda$ is the difference between the red divisor and the
blue divisor, where $\lambda$ is the length of the horizontal
segment.}\label{fig:hopf}
\end{figure}

\begin{ex}\label{ex:hopf}
We consider the tropical surface~$\Delta$ shown in Figure~\ref{fig:hopf}. Note
that while $\Delta$ is topologically a product of a circle with an interval,
we'll see that, as a tropical complex, it does not behave like a product. See
\cite[Sec.~6]{lazar} for a construction of tropical surfaces as products of
curves.

Since $\Delta$ is homotopy equivalent to a circle,
$H^2(\Delta, \ZZ)$ is trivial, and thus Proposition~\ref{prop:ridge-obstruction}
tells us that any divisor on $\Delta$ is algebraically equivalent to a ridge
divisor. By Lemma~\ref{lem:local-intersection-matrix}, ridge divisors are
characterized by having their coefficient vectors in the image of~$M_v$ at each
vertex and one can check that the Cartier ridge divisors on~$\Delta$ are
generated by the top and bottom edges of the cylinder. Moreover, these are
linearly equivalent to each other by the PL function corresponding to the
vertical coordinate in Figure~\ref{fig:hopf}. Thus, $\Pic_{\ridge}(\Delta)$ is
isomorphic to~$\ZZ$. Moreover, one can take a simplicial $1$-cochain
corresponding to a generator of $H^1(\Delta, \ZZ) = \ZZ$ and get a generator for
$\Pic_{\ridge}(\Delta)$, so $\NS(\Delta)$ is trivial. In the notation of
Lemma~\ref{lem:master}, $k=1$ and $m=0$ for $\Delta$.

Thus, using the exponential sequence for ridge divisors~(\ref{eq:exp-ridge}), we
see that $H^0(\Delta, \cD)$ and $H^1(\Delta, \cD)$ are trivial, so the
exponential sequence~(\ref{eq:exp}) gives us an isomorphism $\Pic(\Delta) =
H^1(\Delta, \RR) = \RR$. The divisor corresponding to a real number~$\lambda$ is
depicted in Figure~\ref{fig:hopf} as the difference between the red and blue
segments, where $\lambda$ is the length of the horizontal red segment. Note that
only the length of the horizontal segment, but not its position, affect the
linear equivalence class.
\end{ex}

In the case of a complex projective variety~$X$, the map from $H^1(X, \ZZ)$ into
the complex vector space $H^1(X, \mathcal O_X)$ is always the inclusion of a
lattice, by Hodge theory. The analogous morphism in the tropical exponential
sequence~(\ref{eq:exp}) is $H^0(\Delta, \cD) \rightarrow H^1(\Delta, \RR)$, but
Example~\ref{ex:hopf} shows that the image of this morphism need not be a
lattice. In fact, this failure is related to the non-existence of a divisor with
positive self-intersection on this surface. In particular, we have the
following:

\begin{thm}\label{thm:differentials}
Let $\Delta$ be a tropical surface which is locally connected through
codimension~$1$. Then the $\RR$-span of the image of $H^0(\Delta, \cD)$ in
$H^1(\Delta, \RR)$ has codimension at most~$1$. Moreover, if there exists a
divisor on~$\Delta$ with positive self-intersection, then the $\RR$-span is all
of $H^1(\Delta, \RR)$.
\end{thm}

\begin{proof}
Suppose that the $\RR$-span of the image of $H^0(\Delta, \cD)$ in $H^1(\Delta,
\RR)$ has codimension~$k$. Then, we can find cohomology classes $\sigma_1,
\ldots, \sigma_k$ in $H^1(\Delta, \RR)$ which are linearly independent modulo
the $\RR$-span of $H^0(\Delta, \cD)$. Moreover, by perturbing and then scaling,
we can assume that the $\sigma_i$ are integral classes, i.e.\ in $H^1(\Delta,
\ZZ)$. We look at the images of the $\sigma_i$ in $\Pic_{\ridge}(\Delta)$ by the
map from the exponential sequence~(\ref{eq:exp-ridge}), which generate a
rank~$k$ subgroup by our independence assumption. Thus, we have a
$k$-dimensional subspace of $\Pic_{\ridge}(\Delta) \otimes_{\ZZ} \QQ$ which is algebraically
trivial and thus numerically trivial by Proposition~\ref{prop:intersect-alg}.
Therefore, Lemma~\ref{lem:master} shows that $k \leq 1$, which is the first
statement. Moreover, if $\Delta$ has a divisor with positive self-intersection,
then $m = 1$ in Lemma~\ref{lem:master}, so $k = 0$, meaning that the $\RR$-span
of $H^0(\Delta, \cD)$ is $H^1(\Delta, \RR)$, as desired.
\end{proof}

As a corollary, we have Theorem~\ref{thm:jacobian} from the introduction.

\begin{proof}[Proof of Theorem~\ref{thm:jacobian}]
By Theorem~\ref{thm:differentials}, the $\RR$-span of $H^0(\Delta, \cD)$ is all
of $H^1(\Delta, \RR)$, so it is a lattice in $H^1(\Delta, \RR)$. Therefore,
$H^1(\Delta, \RR) / H^0(\Delta, \cD)$ is isomorphic to compact real torus, which
is isomorphic to the group of algebraically trivial divisors modulo linear
equivalence by the exponential sequence~(\ref{eq:exp}).
\end{proof}

\section{Noether's formula}\label{sec:noether}

In this section, we look at an analogue of Noether's formula for a weak tropical
surfaces~$\Delta$. In particular, we define a rational formal sum of the
vertices of~$\Delta$, which we call the second Todd class of~$\Delta$. The total
degree of the second Todd class is equal to the Euler characteristic of the
underlying $\Delta$-complex.

\begin{defn}\label{def:todd}
The \defi{second Todd class} of a weak tropical surface~$\Delta$ is the
following formal combination of its vertices:
\begin{equation}\label{eq:todd}
\td_2(\Delta) = \frac{1}{12}\sum_{v \in \Delta_0} \Bigg(12 + 5 F_v - 6 E_v -
\sum_{e \in \link_\Delta(v)_0} \alpha(v, e) \Bigg) [v],
\end{equation}
where $F_v$ and $E_v$ are the number of edges and vertices in the link of~$v$,
respectively, and the second summation
is over the vertices of that link.
\end{defn}
While (\ref{eq:todd}) may seem arbitrary, it is the only expression of this form
which satisfies Proposition~\ref{prop:noether} and is determined by a tropical
variety, independent of its subdivision, in the sense established by
Propositions~\ref{prop:invariance} and~\ref{prop:todd-compatibility}.

\begin{proof}[Proof of Proposition~\ref{prop:noether}]
Immediately from the definition, the total degree of the $\td_2(\Delta)$ is:
\begin{equation*}
\frac{1}{12}\sum_{v \in \Delta_{0}} \Bigg(12 + 5 F_v - 6 E_v
- \sum_{e\in\link(v)_0} \alpha(v, e) \Bigg).
\end{equation*}
Since $E_v$ will be count each edge once for each of its two endpoints, and
similarly $F_v$ will count any facet three times, we can rearrange this to:
\begin{equation}\label{eq:todd-degree}
V + \frac{5}{4} F - E -
\frac{1}{12}\sum_{e \in \Delta_1} \big( \alpha(v, e) + \alpha(w, e) \big),
\end{equation}
where $V$, $E$, and~$F$ are the numbers of vertices, edges, and facets
of~$\Delta$ respectively, and $v$ and $w$ are the endpoints of~$e$. For any
edge~$e$, we have the assumption that $\alpha(v, e) + \alpha(w,e) = \degree(e)$
and so the last term of~(\ref{eq:todd-degree}) will be triple counting the faces, once for each edge they
contain. Using this, we get that the degree of the Todd class is $V + F - E =
\chi(\Delta)$, which completes the proof.
\end{proof}

Variations of our definition of the second
Todd class and of Proposition~\ref{prop:noether} have appeared before in the
literature. The first is Kontsevich and Soibelman's
$\ZZ$-affine Gauss-Bonnet theorem, which computes
the Euler characteristic as a sum of local invariants on an oriented manifold
with an $\ZZ$-affine structure away from finitely many
points~\cite[Thm.~2]{kontsevich-soibelman}. If a tropical complex~$\Delta$
happens to be homeomorphic to a manifold, then the sheaf~$\cA$ defined in
Section~\ref{sec:alg} gives $\Delta$ the
structure of a $\ZZ$-affine manifold away from its vertices, and our Todd class
agrees with the local invariant of Kontsevich-Soibelman.

\begin{prop}
If $\Delta$ is homeomorphic to an oriented manifold, then the coefficient of
$\td_2(\Delta)$ at a vertex~$v$ equals the invariant $i_{\mathrm{loc}}(v)$,
defined defined by Kontsevich and Soibelman
in~\cite[Sec.~6.5]{kontsevich-soibelman}.
\end{prop}

\begin{proof}
We fix a vertex~$v$ and look at how the affine linear structure varies as we
make a small loop around~$v$. Let $f$ be a $2$-simplex containing $v$ and let
$u$ and~$w$ be the other vertices of~$f$. We choose coordinates
for~$f$ such that $v = (0,0)$, $u = (1,0)$, and $w=(0,1)$ and look at how
the coordinates change as we cross the edge~$e$ between $v$ and~$w$. Let
$f'$ be the triangle on the other side of~$e$ and $x$ the vertex
of~$f'$ which is not in~$e$. If we extend our affine linear coordinates
across $e$ to $f'$, then $x$ is located at $(-1, \alpha(w, e))$. After applying
the skew transformation
\begin{equation}\label{eq:local-invariant-1}
\begin{pmatrix} 1 & 0 \\ \alpha(w, e) & 1 \end{pmatrix},
\end{equation}
the coordinates of $w$ and~$x$ are $(0, 1)$ and $(-1, 0)$ respectively. To get
the standard affine linear coordinates for $f'$, we make the further change
of coordinates:
\begin{equation}\label{eq:local-invariant-2}
\begin{pmatrix} 0 & 1 \\ -1 & 0 \end{pmatrix}.
\end{equation}
In the notation of~\cite[Sec.~6.5]{kontsevich-soibelman}, the matrices
in~(\ref{eq:local-invariant-1}) and~(\ref{eq:local-invariant-2}) are equal to
$(a_3 a_2^{-1})^{\alpha(w, e)}$ and~$a_2^{-1}$ respectively. From the definition
of the invariant, these contribute $\alpha(w,e)/12$ and $-3/12$ respectively.
Since $e$ has
degree~$2$, we have the relation $\alpha(w, e) = 2 - \alpha(v, e)$, so the sum
of these contributions can be written $-\alpha(v, e) - 1$. After making a complete loop, we've performed one rotation of our coordinate
system, corresponding to the element~$u$ of Kontsevich-Soibelman, which
contributes $1$, and so Kontsevich and Soibelman's local
invariant at~$v$ is:
\begin{equation}\label{eq:local-invariant}
1 - \frac{1}{12} \sum_{e \in \link_{\Delta}(v)_{0}} \big(\alpha(v, e) + 1\big).
\end{equation}

On the other hand, if $\Delta$ is a manifold, then the link of~$v$ is a cycle,
so it has the same number of edges as vertices. In this case, the coefficient
of~$v$ in the second Todd class according to~(\ref{eq:todd}) simplifies
to~(\ref{eq:local-invariant}) and so we're done.
\end{proof}

Another precursor to our definition of the second Todd class is 
the second Chern class of a tropical manifold as defined by Kristin
Shaw~\cite[Def.~3.2.14]{shaw}. A tropical manifold is locally modeled on the
Bergman fan of a loopless matroid~\cite{ardila-klivans}.  In the rest of this
section, we study the translation of Definition~\ref{def:todd}, not just for
matroid fans, but any multiplicity-free tropical variety.
In particular, by \defi{multiplicity-free tropical variety}, we mean a subset
of~$\RR^N$ which is the support of a balanced, $2$-dimensional finite polyhedral
complex, where all facets have multiplicity~$1$. 

The local structure of the tropical variety~$V$ at a point~$p$ is given by its
\defi{star} $\operatorname{star}_V(p)$, which consists of points $w \in \RR^N$ such that $p + \varepsilon w
\in V$ for all sufficiently small~$\varepsilon$. We now describe the second Todd class
of~$V$ at~$p$ in terms of its star. We first choose a unimodular fan~$\Sigma$
whose support is $\operatorname{star}_V(p)$. Let $v$ be the minimal
vertex along a ray~$e$ of~$\Sigma$ and let $w_1, \ldots, w_d$ be the
minimal vertices along rays adjacent to~$v$. By the balancing condition and
since $v$ is the minimal integral vertex,
\begin{equation}\label{eq:star-relation}
w_1 + \cdots + w_d = c v
\end{equation}
for some integer~$c$ and we define $\alpha(0, e) = d - c$, where the notation is
taken to be suggestive of the structure constants of a weak tropical complex
coming from a subdivision of a tropical
variety~\cite[Const.~3.3]{cartwright-complexes}.
Then, we define the second Todd
class of $\Sigma$ to be:
\begin{equation}\label{eq:todd-fan}
\td_2(\Sigma) = \frac{1}{12} \Bigg(12 + 5 F - 6 E - \sum_{e \in \Sigma_1}
\alpha(0, e) \Bigg),
\end{equation}
where $F$ and $E$ are the number of $2$-dimensional cones and rays of~$\Sigma$,
respectively, and the summation is over the rays of~$\Sigma$.

\begin{prop}\label{prop:invariance}
If $\Sigma$ and $\Sigma'$ are two unimodular fans with the same support, then
$\td_2(\Sigma) = \td_2(\Sigma')$.
\end{prop}

\begin{proof}
We let $\Sigma''$ be a unimodular common refinement of $\Sigma$ and~$\Sigma'$,
which exists by toric resolution of singularities~\cite[Sec.~2.6]{fulton-toric}.
As in the proof of Lemma~\ref{lem:local-matrix}, the theory of
minimal models for toric varieties shows that $\Sigma''$ is formed by
subdivisions of $\Sigma$ and similarly for $\Sigma'$, where a subdivision means
replacing a cone spanned by rays $r_1$ and~$r_2$ with two cones, spanned by
$r_1$ and $r_1 + r_2$ and by $r_1+r_2$ and $r_2$, respectively. Thus, it
suffices to prove the theorem when $\Sigma'$ is the fan formed from $\Sigma$ by
such a subdivision.

If we let $\alpha'$ denote the parameters for $\Sigma'$ after such a
subdivision, then $\alpha'(0,r) = \alpha(0,r)$ except for:
\begin{equation*}
\alpha'(0, r_1) = \alpha(0, r_1) -1 \quad \alpha'(0, r_1 + r_2) = 1
\quad \alpha'(0, r_2) = \alpha(0, r_2) - 1.
\end{equation*}
Thus, under the subdivision, $\sum \alpha(0,e)$ in the definition of the second
Todd class~(\ref{eq:todd-fan}) decreases by~$1$, so the second Todd class
of~$\Sigma'$ is:
\begin{align*}
\td_2(\Sigma') &=
\frac{1}{12} \Bigg( 12 + 5(F+1) - 6(E+1) - \Big({-1} + \sum_{e \in
\Sigma_1} \alpha(0, e)\Big) \Bigg) \\
&=
\frac{1}{12} \Bigg(12 + 5F - 6E - \sum_{e \in \Sigma_1} \alpha(0, e) \Bigg)
= \td_2(\Sigma) \qedhere
\end{align*}
\end{proof}

By Proposition~\ref{prop:invariance}, we can define a second Todd class for any
tropical variety~$V$. Specifically, we define $\td_2(V) = \sum_{p \in
V}\td_2(\Sigma_p)[p]$, where for each point~$p$, $\Sigma_p$ is any unimodular
fan supported on $\operatorname{star}_V(p)$. The formal sum is finite because if
we choose a polyhedral decomposition of~$V$, then $\td_2(V)$ is supported at its
vertices by the following result.

\begin{prop}\label{prop:todd-triviality}
If $\Sigma$ is a unimodular fan whose support is a product with~$\RR$, then
$\td_2(\Sigma)$ is trivial.
\end{prop}

\begin{proof}
We choose coordinates such that the factor of~$\RR$ is the first coordinate.
Then, by Proposition~\ref{prop:invariance}, we can replace~$\Sigma$ with a fan
whose rays are:
\begin{align*}
r_1 &= \RR_{\geq 0} \cdot (1, 0, \ldots, 0) \\
r_2 &= \RR_{\geq 0} \cdot (-1, 0, \ldots, 0) \\
r_3, \ldots, r_n &\subset \{0\} \times \RR^{N-1}
\end{align*}
and whose $2$-dimensional cones are spanned by every combination of either $r_1$
or~$r_2$ and one of $r_3, \ldots, r_n$. Then, for the definition of the second
Todd class, $E = n$, $F= 2(n-2)$, and we can
compute the constants $\alpha(0,r_i)$ as follows:
\begin{align*}
\alpha(0,r_1) = \alpha(0, r_2) &= n-2 &
\alpha(0, r_3) = \cdots = \alpha(0, r_n) &= 2.
\end{align*}
Plugging these into~(\ref{eq:todd-fan}), we get:
\begin{equation*}
\td_2(\Sigma) = \frac{1}{12}
\big(12 + 10 (n-2) - 6 n - 2(n-2) - (n-2) 2\big) = 0,
\end{equation*}
which completes the proof.
\end{proof}

Our definitions of the second Todd class for multiplicity-free tropical
varieties and weak tropical complexes are compatible in the following sense.
In~\cite[Sec.~3]{cartwright-complexes}, a recipe was given for converting the
tropicalization of a sch\"on algebraic variety, with a unimodular subdivision,
into a weak tropical complex. The algebraic variety was necessary because cells
from the subdivision could be duplicated in the construction of the
parametrizing tropical variety introduced by~\cite{helm-katz}. Since we're
working with multiplicity-free tropical varieties in a combinatorial setting, we
do not have initial ideals, and we can construct a weak tropical complex without
any duplication, using exactly the bounded cells of the subdivision and the
structure constants as in~\cite[Const.~3.3]{cartwright-complexes}. Then, we have
the following compatibility:

\begin{prop}\label{prop:todd-compatibility}
Let $V$ be a tropical variety with a unimodular subdivision as
in~\cite[Sec.~3]{cartwright-complexes} and let $\Delta$ the
weak tropical surface formed from the bounded cells. Then, $td_2(V)$ and
$\td_2(\Delta)$ agree at all points not in the closure of the unbounded cells
of~$V$.
\end{prop}

\begin{proof}
Let $p$ be a point of~$V$ not contained in the closure of any unbounded cell
of the subdivision, and we wish to show that the coefficients of $\td_2(V)$ and
$\td_2(\Delta)$ agree at~$p$. First suppose that $p$ is not a vertex of the
subdivision. By definition, $\td_2(\Delta)$ is supported at the vertices
of~$\Delta$, so its coefficient is trivial at~$p$. Likewise, in a neighborhood
of~$p$, $V$ can be factored as a product with~$\RR$, so $\td_2(V)$ is also
trivial at~$p$ by Proposition~\ref{prop:todd-triviality}.

Thus, we are reduced to the case where $p$ is a vertex of the unimodular
subdivision. In this case, the subdivision of~$V$ decomposes $\link_V(p)$ into a
unimodular fan. Since $p$ is not contained in any unbounded cells, the cones of
this fan are in natural correspondence with cells of $\link_p(\Delta)$. By the
formal similarity of the equations for the second Todd class in the two
settings, it only remains to check that if $e$ is an edge with
endpoint~$p$, then $\alpha(p,e)$ from the weak tropical complex equals
$\alpha(0,e)$ defined from the fan. Both constructions involve a sum of the
vertices of facets containing~$e$, with the only difference being
that the construction of the structure constant $\alpha(p,e)$ works in
$\RR^{N+1}$ with $V$ placed in $\RR^{N} \times \{1\}$. The star of~$V$ at~$p$
can be obtained as the quotient of $\RR^{N+1}$ by the line generated by~$p$ in
this embedding. Taking this quotient, the relation from
\cite[Const.~3.3]{cartwright-complexes} becomes~(\ref{eq:star-relation}), and so
$c$ in the latter equation equals $\alpha(w, e)$, where $w$ is the endpoint
of~$e$ in~$\Delta$ other than $p$. Thus, $\alpha(v,e) = d - c$, where $d$ is the
degree of~$e$, which is the same as our definition of $\alpha(0, r)$.
\end{proof}

Now we return to Bergman fans of rank~$3$ matroids, for which we can express the
second Todd class more explicitly in terms of the matroid's invariants.
In addition, we will
compute the square of the canonical divisor in order to relate our definition of
the second Todd class with the second Chern class of~\cite{shaw}. As
in~\cite[Sec.~5.3]{mikhalkin}, we define the canonical divisor of
multiplicity-free tropical variety~$V$ to be $K_V = \sum_r (\deg e - 2) [r]$,
where the summation is over the $1$-dimensional cells, for any choice of
polyhedral decomposition of~$V$.

\begin{lem}\label{lem:matroid-invariants}
Let $V$ be the support of the Bergman fan of a rank~$3$ simple matroid~$M$ with
$n$ elements, $m$ flats, and $\ell$ complete flags. Then, $K_V$ is a divisor and
we have the following invariants:
\begin{align*}
K_V^2 &= (-5n - 4m + 3 \ell +9) [0] \\
\td_2(V) &= {\textstyle\frac{1}{12}}(-7n -5m +4\ell + 12) [0],
\end{align*}
\end{lem}

\begin{proof}
By \cite[Prop.~3.9]{cartwright-complexes}, we can
compute $K_V^2$ using the tropical complex~$\Delta$ of any subdivision
of~$V$ such that only bounded cells contain the origin. In addition, we choose a
subdivision which agrees with the fine subdivision in a neighborhood of the
origin.
We will now compute the local intersection matrix of
$\Delta$ at the origin. Recall from~\cite{ardila-klivans} that $V$ is
constructed in $\RR^{n-1}$, whose integral points are generated by vectors $v_A$
as $A$ ranges over the elements of~$M$, and these vectors satisfy the single
relation $\sum_A v_A = 0$. The rays of the fine
subdivision of~$V$, and thus the edges of~$\Delta$ containing $0$, are in
bijection with the disjoint union of the elements of~$M$ and its rank~$2$ flats.
The ray for an element~$A$ of~$M$ is spanned by $v_A$ and the ray of a rank~$2$
flat~$S$ is spanned by~$v_S$, which is defined to be $\sum_{A \in S} v_A$. In
both cases, $v_A$ and~$v_S$ are the first integral points along their rays, so
they are vertices of~$\Delta$. For every element~$A$ contained in a rank~$2$
flat~$S$, we have a face of~$\Delta$ containing $0$, $v_A$, and~$v_S$.

We now compute the structure constants for~$\Delta$. If $A$ is any element
of~$M$, then the facets containing the edge~$e_A$ from $0$ to $v_A$ are in
bijection with the rank~$2$ flats containing~$A$. If $b_A$ denotes the number of
rank~$2$ flats containing~$A$, then
\begin{equation*}
\sum_{S \ni A} v_S = b_A v_A + \sum_{B \neq A} v_B = (b_A - 1) v_A,
\end{equation*}
where, in the second summation, $B$ is any element of~$M$, and the first
equality is because
every element $B \neq A$ is in a unique rank~$2$ flat with~$A$.
Thus, $\alpha(v_A, e_A) = b_A - 1$ and
$\alpha(0, e_A) = 1$.
Similarly, we let $S$ be a rank~$2$ flat, and the faces containing the
edge~$e_S$ from $0$ to~$v_S$ correspond to the elements of~$S$, for which we
have $\sum_{A \in S} v_A = v_S$, by definition. Thus, $\alpha(v_S, e_S) = 1$ and
$\alpha(0, e_S) = \# S - 1$.

We number the elements $A_1, \ldots, A_n$ of~$M$
and the rank~$2$ flats $S_1, \ldots, S_m$ and we order the edges of~$\Delta$
containing $0$ as $e_{A_1}, \ldots, e_{A_n}, e_{S_1}, \ldots, e_{S_m}$. Then, we
have the local intersection matrix for~$\Delta$:
\begin{equation}\label{eq:local-matroid}
M_0 = \begin{pmatrix}
-b_{A_1} + 1 & \cdots & 0 & * & \cdots & * \\
\vdots & \ddots & \vdots & \vdots & & \vdots \\
0 & \cdots & -b_{A_n} + 1 & * & \cdots & * \\
* & \cdots & * & -1 & \cdots & 0 \\
\vdots & & \vdots & \vdots & \ddots & \vdots \\
* & \cdots & * & 0 & \cdots & -1
\end{pmatrix},
\end{equation}
where the $*$ denote blocks which record the incidences between the elements and
rank~$2$ flats.

By Lemma~\ref{lem:local-intersection-matrix}, to show that $K_V$ is a divisor,
it is sufficient to show that its vector representation is in the image of
$M_0$. This vector representation is:
\begin{equation*}
[K_V] = \begin{cases}
b_{A_i} - 2 & \mbox{if } i \leq n \\
\# S_{i-n} - 2 & \mbox{if } i > n
\end{cases}
\end{equation*}
Using the description~(\ref{eq:local-matroid}) of $M_0$, one can check that
$[K_V] = M_0 \bff$, where $\bff$ is defined by:
\begin{equation*}
\bff_i = \begin{cases}
-2 &\mbox{if $i = 1$} \\
1 & \mbox{if $2 \leq i \leq n$}\\
-1 & \mbox{if $i > n$ and the flat $S_{i-n}$ contains the element $A_1$} \\
2 & \mbox{if $i > n$ and the flat $S_{i-n}$ does not contain the element $A_1$}.
\end{cases}
\end{equation*}

The coefficient of~$K_V^2$ at the origin is given by the product $\bff^T M_0
\bff$ by Proposition~\ref{prop:matrix-intersections}. Using the definition of $\bff$ and of $[K_v] = M_0 \bff$, we
have:
\begin{align*}
K_V^2 &= -2(b_{A_1} - 2) + \sum_{i=2}^n (b_{A_i} - 2)
- \sum_{\substack{i = 1 \\ S_i \ni A_1}}^m (\# S_i - 2 )
+ 2 \sum_{\substack{i = 1 \\ S_i \not\ni A_1}}^m (\# S_i - 2) \\
&= -3(b_{A_1} - 2) + \sum_{i=1}^n (b_{A_i} - 2)
- 3\sum_{\substack{i=1 \\ S_i \ni A_1}}^m (\# S_i - 1 - 1)
+ 2 \sum_{i = 1}^m (\# S_i - 2) \\
\intertext{Every element other than $A_1$ is contained in exactly one flat with
$A_1$, so $\#S_i - 1$ in the second summation counts the elements other than
$A_1$, and with this we can evaluate the sums:}
&= -3b_{A_1} + 6 + (\ell - 2 n) - 3(n - 1 - b_{A_{1}}) + 2 (\ell - 2 m) \\
&= -5n -4m + 3 \ell + 9,
\end{align*}
which proves the first claim.

Second, we compute the second Todd class.
From our computations of the structure constants above, we can compute the sums:
\begin{equation*}
\sum_{A} \alpha(0, e_A) = n   \qquad\qquad
\sum_{S} \alpha(0, e_S) = \ell - m.
\end{equation*}
Now, we use the definition of the second Todd class~(\ref{eq:todd}):
\begin{align*}
\td_2(\Delta) &= {\textstyle\frac{1}{12}}(12 + 5 \ell - 6 (n + m) -
(n + \ell - m))[0] \\
&= {\textstyle\frac{1}{12}}(12 - 7n - 5m - 4\ell )[0],
\end{align*}
which is the second formula from the lemma statement.
\end{proof}

\begin{cor}
If $V$ is a matroid fan, then we have the relation $12 \td_2(V) = K_V^2 +
c_2(V)$, where $c_2(V)$ is the second Chern class, as defined by Shaw.
\end{cor}

\begin{proof}
The second Chern class of a matroid fan in~\cite[Def.~3.2.13]{shaw} is, in the
notation of Lemma~\ref{lem:matroid-invariants}, equal to to $3 - 2n - m + \ell$,
from which the claimed equality is immediate.
\end{proof}

\end{document}